\documentclass[11pt]{amsart}
\usepackage{fullpage}
\usepackage[colorlinks=true,pagebackref=false,hyperindex,citecolor=green,linkcolor=blue, urlcolor=blue]{hyperref}
\usepackage{amsmath}
\usepackage{amsfonts}
\usepackage{amssymb}
\usepackage{color}
\usepackage{mathrsfs}
\usepackage{mdwlist} 
\usepackage[active]{srcltx} 
\usepackage{stmaryrd}
\usepackage{palatino}

\theoremstyle{definition}
\newtheorem{theorem}{Theorem}[section]

\newtheorem{setup}[theorem]{Setup}
\newtheorem{corollary}[theorem]{Corollary}
\newtheorem{lemma}[theorem]{Lemma}
\newtheorem{proposition}[theorem]{Proposition}

\newtheorem{notation}[theorem]{Notation}
\newtheorem{convention}[theorem]{Convention}
\newtheorem{question}[theorem]{Question}

\newtheorem*{theoremA}{Theorem A}
\newtheorem*{theoremB}{Theorem B}
\newtheorem*{theoremC}{Theorem C}
\newtheorem*{theoremD}{Theorem D}
\newtheorem*{theoremE}{Theorem E}

\newtheorem*{MainR}{Theorem \ref{Main: T}}
\newtheorem*{UniformLBoundsR}{Corollary \ref{UniformLBounds: C}}
\newtheorem*{SecondaryR}{Theorem \ref{Secondary: T}}

\theoremstyle{definition}
\newtheorem{definition}[theorem]{Definition}
\newtheorem{example}[theorem]{Example}
\newtheorem{conjecture}[theorem]{Conjecture}
\newtheorem{remark}[theorem]{Remark}
\numberwithin{equation}{subsection}

\newcommand{\Colon}[2]{\left( #1 \colon #2 \right)}
\newcommand{\bracket}[2]{{#1}^{[p^{#2}]}}
\newcommand{\jac}[1]{\operatorname{Jac}\( #1 \)}
\newcommand{\m}{\mathfrak{m}}

\newcommand{\wt}[1]{\omega_{#1}}  
\newcommand{\qdeg}{\operatorname{deg}}
\newcommand{\qh}{homogeneous }
\newcommand{\qhns}{homogeneous} 
\newcommand{\Deg}[2]{\left[ #1 \right]_{#2}}

\newcommand{\up}[1]{\left\lceil #1 \right\rceil}
\newcommand{\down}[1]{\left\lfloor #1 \right\rfloor}
\newcommand{\tr}[2]{\left \langle {#1} \right \rangle_{#2}} 
 
\newcommand{\digit}[2]{ {#1}^{(#2)}}
\newcommand{\base}{\operatorname{base}}

\newcommand{\lpr}[2]{ \left \llbracket \hspace{.01in} #1 \ \% \ #2 \hspace{.01in} \right \rrbracket} 
\newcommand{\paren}[2]{#1 #2}
\newcommand{\parenone}[1]{#1}

\renewcommand{\th}{\text{th}}

\newcommand{\new}[1]{ \nu_{f}(p^{#1})} 
\newcommand{\newfg}[1]{ \nu_{f+g}(p^{#1})} 
\newcommand{\fpt}[1]{\boldsymbol{\operatorname{fpt}}\left(#1\right)}
\newcommand{\lct}[1]{\boldsymbol{\operatorname{lct}}\left(#1\right)}

\newcommand{\set}[1]{\left\{ #1 \right\}}

\newcommand{\df}[1]{\partial_{{#1}} (f)}  
\renewcommand{\(}{\left(}
\renewcommand{\)}{\right)}

\newcommand{\sM}{\mathscr{M}}

\newcommand{\Ell}{L}

\newcommand{\order}[2]{\operatorname{ord}(#1,#2)}

\newcommand{\FF}{\mathbb{F}}
\newcommand{\NN}{\mathbb{N}}
\newcommand{\LL}{\mathbb{L}}
\newcommand{\ZZ}{\mathbb{Z}}
\newcommand{\RR}{\mathbb{R}}

\newcommand{\Q}{\mathbb{Q}}
\newcommand{\N}{\mathbb{N}}

\newcommand{\fp}{f_p}
\newcommand{\fQ}{f_{\mathbb{Q}}}

\begin{document}

\title{$F$-pure thresholds of homogeneous polynomials}
\author{Daniel J.\ Hern\'andez, Luis N\'u\~nez-Betancourt, Emily E.\ Witt, and Wenliang Zhang}
\maketitle

\begin{abstract}
In this article, we investigate $F$-pure thresholds of polynomials that are homogeneous under some $\N$-grading, and have an isolated singularity at the origin.  We characterize these invariants in terms of the base $p$ expansion of the corresponding log canonical threshold. 
As an application, we are able to make precise some bounds on the difference between $F$-pure and log canonical thresholds established by Musta\c{t}\u{a} and the fourth author.  
We also examine the set of primes for which the $F$-pure and log canonical threshold of a polynomial must differ.
Moreover, we obtain results in special cases on the ACC conjecture for $F$-pure thresholds, and on the upper semi-continuity property for the $F$-pure threshold function.  
\end{abstract}

\section{Introduction}
\newtheorem*{theoremNL}{Theorem}
\newtheorem*{corollaryNL}{Corollary}
\newtheorem*{propositionNL}{Proposition}
\newtheorem*{questionNL}{Question}

The goal of this article is to investigate $F$-pure thresholds, and further study their relation with log canonical thresholds. The $F$-pure threshold, first defined in \cite{TW2004}, is an invariant of singularities in positive characteristic defined via splitting conditions and the Frobenius (or $p^{\th}$-power) endomorphism.  Though $F$-pure thresholds may be defined more generally, we will only consider $F$-pure thresholds of polynomials over fields of prime characteristic, and thus follow the treatment given in \cite{MTW2005}. 
Given such a polynomial $f$, the $F$-pure threshold of $f$, denoted by $\fpt{f}$, is always a rational number in $(0,1]$, with smaller values corresponding to ``worse'' singularities \cite{BMS2008, BMS2009, BSTZ2009}.

The log canonical threshold of a polynomial $\fQ$ over $\Q$, denoted $\lct{\fQ}$, is an important invariant of singularities of $\fQ$, and can be defined via integrability conditions, or more generally, via resolution of singularities.  
Like the $F$-pure threshold, $\lct{\fQ}$ is also a rational number contained in $(0,1]$;  see \cite{BL2004} for more on this (and related) invariants.  In fact, the connections between $F$-pure and log canonical thresholds run far deeper:  As any $\frac{a}{b} \in \Q$ determines a well-defined element of $\FF_p$ whenever $p \nmid b$, we may reduce the coefficients of $\fQ$ modulo $p \gg 0$ to obtain polynomials $\fp$ over $\FF_p$.  Amazingly, the $F$-pure thresholds of these so-called characteristic $p$ models of $\fQ$ are related to the log canonical threshold of $\fQ$ as follows  \cite[Theorems 3.3 and 3.4]{MTW2005}: 
\begin{equation}
\fpt{\fp} \leq \lct{\fQ} \text{ for all } p \gg 0 \text{ and } \lim \limits_{p \to \infty} \fpt{\fp} = \lct{\fQ}. \label{relations}
\end{equation}
In this article, we will not need to refer to the definition of $\lct{\fQ}$ via resolutions of singularities, and instead take the limit appearing in \eqref{relations} as our definition of $\lct{\fQ}$. Via reduction to characteristic $p>0$, one may reduce
polynomials (and more generally, ideals of finite type algebras) over
\emph{any} field of characteristic zero to characteristic $p \gg 0$ (e.g.,
see \cite{Karen1997}).  Moreover, the relations in \eqref{relations} are just two of 
many deep connections between invariants of characteristic $p$
models defined via the Frobenius endomorphism, and invariants of the
original characteristic zero object that are often defined via resolution of
singularities.  For more in this direction, see, for example,
\cite{MTW2005, BMS2006, Karen2000, Karen1997b, Hara1998, HW2002,
  HY2003, Tak2004, Schwede2007, BST13, STZ12}.

Motivated by the behavior exhibited when $\fQ$ defines an elliptic curve over $\Q$, it is conjectured  that for \emph{any} polynomial $\fQ$ over $\Q$, there exist infinitely many primes for which $\fpt{\fp}$ equals $\lct{\fQ}$ \cite[Conjecture 3.6]{MTW2005}. This conjecture, along with other characteristic zero considerations, has fueled interests in understanding various properties of $\fpt{\fp}$.  
In particular,  arithmetic properties of the denominator of $\fpt{\fp}$ have recently been investigated, most notably by Schwede (e.g., see \cite{Schwede2008}).   
Assuming $\fpt{\fp} \neq \lct{\fQ}$, Schwede has asked when $p$ must divide the denominator of $\fpt{\fp}$, and the first author has asked when the denominator of $\fpt{\fp}$ must be a power of $p$, and more specifically, when $\fpt{\fp}$ must be a \emph{truncation} of $\lct{\fQ}$.\footnote{\url{https://sites.google.com/site/computingfinvariantsworkshop/open-questions}}
Recall that, given the unique non-terminating (base $p$) expansion $\lct{\fQ} = \sum_{e \geq 1} \digit{\lambda}{e} \cdot p^{-e} \in (0,1]$, we call $\fpt{\fp}$ a truncation of $\lct{\fQ}$ (base $p$) if $\fpt{\fp} = \sum_{e=1}^L \digit{\lambda}{e} \cdot p^{-e}$ for some $L \geq 1$.

In this paper, we study $F$-pure thresholds associated to polynomials that are homogeneous under some (possibly, non-standard) $\NN$-grading, and that have an isolated singularity.  The $F$-purity of such polynomials was originally investigated by Fedder (e.g., see \cite[Lemma 2.3 and Theorem 2.5]{Fed1983}), and more recently, by Bhatt and Singh, who showed the following:  Given a (standard-graded) homogeneous polynomial  $f$ over $\FF_p$ of degree $n$ in $n$ variables with an isolated singularity at the origin, if $p \gg 0$, then $\fpt{\fp} = 1 - \frac{A}{p}$ for some integer $0 \leq A \leq n-2$.  Bhatt and Singh also show that, if $f$ is (standard-graded) homogeneous of arbitrary degree with an isolated singularity at the origin, and if $\fpt{\fp} \neq \lct{\fQ}$, then the denominator of $\fpt{\fp}$ is a power of $p$ whenever $p \gg 0$ \cite[Theorem 1.1 and Proposition 5.4]{BhattSingh}. 

We combine a generalization of the methods in \cite{BhattSingh} with a careful study of base $p$ expansions to obtain our main result, Theorem \ref{Main: T}, which characterizes $F$-pure thresholds of  polynomials with an isolated singularity at the origin that are \qh under {some} $\NN$-grading.
Our result states  that such $F$-pure thresholds must have a certain (restrictive) form; 
in particular, it confirms that the denominator of $\fpt{\fp}$ is a power of $p$ whenever $\fpt{\fp} \neq \lct{\fQ}$ for this larger class of polynomials.
Notably, the result also gives a bound for the power of $p$ appearing in the denominator of $\fpt{\fp}$ for $p \gg 0$.  To minimize technicalities, we omit the statement of Theorem \ref{Main: T}, and instead discuss the two variable case, where our main result takes the following concrete form;  note that in what follows, we use $\jac{f}$ to denote the ideal generated by the partial derivatives of a polynomial $f$, and $\order{p}{b}$ to denote the least positive integer $k$ such that $p^k \equiv 1 \bmod b$.

\begin{theoremA}[{cf.}\ Theorem \ref{Main2D: T}]  \label{A: Thm}
Fix an $\NN$-grading on $\mathbb{F}_p[x,y]$, and consider a \qh polynomial $f$ with $\sqrt{\jac{f}} = (x,y)$ such that $\qdeg f \geq \qdeg xy$.  If $p \nmid \qdeg f$ and $\fpt{f} \neq \frac{\qdeg xy}{\qdeg f}$, then \[ \fpt{f} = \frac{ \qdeg xy}{\qdeg f} - \frac{\lpr{\paren{p^L}{\qdeg xy}}{\qdeg f }}{p^L\qdeg f } \text{ for some integer $1 \leq L \leq \order{p}{\qdeg f}$,} \] 
where $\lpr{\paren{a}{p^L}}{b}$ denotes the least \emph{positive} residue of $a p^L$ modulo $b$.
\end{theoremA}

In fact, we are able to give a slightly more refined description of the $F$-pure threshold, even in the two variable case;  we refer the reader to Theorem \ref{Main2D: T} for the detailed statement.  Moreover, we may recast Theorem A as a theorem relating $F$-pure and log canonical thresholds:  If $\fQ \in \Q[x,y]$ is a homogenous and satisfies the conditions appearing in Theorem $A$ (i.e., $\qdeg \fQ \geq \qdeg xy$ and $(x,y) = \sqrt{\jac{\fQ}}$), then it is well-known (e.g., see Theorem \ref{DifferenceBounds: T}) that $\lct{\fQ} = \frac{ \qdeg xy}{\qdeg f}$. Substituting this identity into Theorem A leads to a description of $\fpt{\fp}$ in terms of $\lct{\fQ}$, and in fact is enough to show that $\fpt{\fp}$ is a truncation of $\lct{\fQ}$ (e.g., see Lemma \ref{explicittruncation: L}).

Though the situation is more subtle, many of the properties highlighted by Theorem A and the subsequent discussion hold in general (after some slight modifications); we refer the reader to Theorem \ref{Main: T} for a detailed description of $F$-pure thresholds in higher dimensions.  Moreover, motivated by (the bounds for $L$ appearing in) Theorem A, one may ask whether there always exists a (small) finite list of possible values for $F$-pure thresholds, say, as a function of the class of $p$ modulo $\qdeg f$.  This question turns out to have a positive answer for homogeneous polynomials with isolated singularities.  Furthermore,  these lists can be minimal, and  strikingly, can even precisely determine $\fpt{\fp}$.  For examples of such lists, see Examples \ref{6LinearFactors: E}, \ref{7LinearFactors: E}, and \ref{determined: E}.

The remaining results in this article are all applications of our description of $F$-pure thresholds.  The first such application concerns uniform bounds for the difference between log canonical and $F$-pure thresholds. 
We recall the following result, due to Musta\c{t}\u{a} and the fourth author:
Given a polynomial $\fQ$ over $\Q$, there exist constants $C \in \mathbb{R}_{>0}$ and $N \in \NN$ (depending only on $\fQ$) such that 
\[ \frac{1}{p^N} \leq \lct{\fQ} - \fpt{\fp} \leq \frac{C}{p} \] whenever $\fpt{\fp} \neq \lct{\fQ}$ and $p \gg 0$ \cite[Corollaries 3.5 and 4.5]{MZ2012}.  We stress that the preceding result applies to an arbitrary polynomial, and that the constants $C$ and $N$ are not explicitly stated as functions of $\fQ$.  In the special case of a \qh polynomial with an isolated singularity at the origin, we give a new proof of this result that makes explicit one optimal choice of constants.

\begin{theoremB}[ {cf}.\ Theorem \ref{DifferenceBounds: T}]
Suppose  $\fQ \in \Q[x_1, \cdots, x_n]$ is \qh under some $\NN$-grading with an isolated singularity at the origin, and write the rational number $\lct{\fQ} = \frac{a}{b}$ in lowest terms.  If $p \gg 0$, then either $\fpt{\fp} = \lct{\fQ}$, or \[\frac{b^{-1}}{p^{\order{p}{b}}}\leq \lct{\fQ}- \fpt{\fp} \leq \frac{n-1-b^{-1}}{p}.\]
Moreover, these bounds are sharp (see Remark \ref{SharpBounds: R}).
\end{theoremB}

Much of the focus of this article is on studying the form of the $F$-pure threshold when it differs from the log canonical threshold.  In Section \ref{BadPrimes: SS}, we give a simple criterion that, when satisfied, guarantees that the $F$-pure and log canonical threshold must differ.  {The main result of this section, Proposition \ref{BadPrimes: P}, holds quite generally; that is, it can be applied to polynomials that are neither \qhns, nor have an isolated singularity.   Moreover, the proof of this result is elementary, and is based upon the fact that the base $p$ expansion of an $F$-pure threshold must satisfy certain rigid conditions, as was observed in \cite{BMS2009, Singularities}}

\begin{theoremC}[{cf.}\ Proposition \ref{BadPrimes: P}]
Let $\fQ$ denote \emph{any} polynomial over $\Q$, and write $\lct{\fQ} =\frac{a}{b}$ in lowest terms.  If $a \neq 1$, then the set of primes for which $\lct{\fQ}$ is not an $F$-pure threshold (of \emph{any} polynomial) is infinite, and contains all primes $p$ such that $p^e \cdot a \equiv 1 \bmod b$ for some $e \geq 1$.  In particular, the density of the set of primes $\set{ p : \fpt{f_p} \neq \lct{\fQ}}$ is greater than or equal to $\frac{1}{\phi(b)}$, where $\phi$ denotes Euler's phi function.
\end{theoremC}

As a further application of our main theorem, we are {also} able to construct a large class of polynomials $\fQ$ over $\Q$ for which the density of the set $\set{ p : \fpt{\fp} \neq \lct{\fQ}}$ is larger than any prescribed bound between zero and one.

\begin{theoremD}[{cf.}\ Example \ref{LargeClassBadPrimes: E}] 
For every $\varepsilon > 0$, there exists an integer $n$ with the following property:  If $\fQ \in \Q[x_1,\ldots, x_{n-1}]$ is homogeneous (under the standard grading) of degree $n$ with an isolated singularity at the origin, then the density of the set of primes $\set{ p : \fpt{f_p} \neq  \lct{\fQ}}$ 
is greater than $1 - \varepsilon$.
\end{theoremD}

The remaining applications deal with another connection between $F$-pure and log canonical thresholds:  Motivated by results in characteristic zero, it was conjectured in \cite[Conjecture 4.4]{BMS2009} that the set of \emph{all} $F$-pure thresholds of polynomials in a (fixed) polynomial ring over a field of characteristic $p>0$ satisfies the ascending chain condition (ACC), i.e., contains no strictly increasing sequences.  In Proposition \ref{ACC: P}, we prove that a restricted set of $F$-pure thresholds satisfies ACC.  Though the characteristic zero analog of Proposition \ref{ACC: P} (that is, the statement obtained by replacing ``$\mathbb{F}_p$'' with ``$\Q$'' and ``$F$-pure threshold'' with ``log canonical threshold," as appropriate) is obvious, our result relies strongly on the description of $F$-pure thresholds given in Theorem \ref{Main: T}.

Finally, as detailed in \cite[Remark 4.5]{BMS2009}, the ACC conjecture for $F$-pure thresholds predicts that for any polynomial $f \in \mathbb{F}_p[x_1, \cdots, x_n]$, there exists an integer $N$ (which may depend on $f$) such that $\fpt{f} \geq \fpt{f+g}$ for all $g \in (x_1, \cdots, x_n)^N$.  In our final application, we are able to confirm this prediction in the following special case.

\begin{theoremE}[{cf}.\ Proposition \ref{madicusc: P}]
Suppose that  $f \in \mathbb{F}_p[x_1, \cdots, x_n]$ is \qh under some $\NN$-grading such that $\sqrt{\jac{f}} =(x_1, \ldots, x_n)$ and ${\qdeg f}\geq {\sum \qdeg x_i}$.  Then $\fpt{f} = \fpt{f + g}$ for each $g\in(x_1, \cdots, x_n)^{n \qdeg f - \sum \qdeg x_i+1}$.
\end{theoremE}

\subsection*{Notation}  Throughout this article, $p$ denotes a prime number and $\mathbb{F}_p$ denotes the field with $p$ elements.  For every ideal $I$ of a ring of characteristic $p>0$, and every $e \geq 1$, $\bracket{I}{e}$ denotes the $e^{\th}$ Frobenius power of $I$, the ideal generated by the set $\set{ g^{p^e} : g \in I}$.  
For a real number $a$, $\lceil a \rceil$ (respectively, $\lfloor a \rfloor$) denotes the least integer that is greater than or equal to (respectively, greatest integer less or equal to) $a$. 

\subsection*{Acknowledgements}  
The authors are indebted to Bhargav Bhatt and Anurag Singh; the latter shared ideas on their joint work during the Midwest Commutative Algebra and Geometry Conference at Purdue University in 2011 that would eventually form the foundation of our approach.
We would also like to thank Benjamin Weiss and Karen Smith for their comments on an earlier draft.  The first author gratefully acknowledges support from the Ford Foundation (FF) through a FF Postdoctoral Fellowship.  The second author thanks the National Council of Science and Technology (CONACyT) of Mexico for support through Grant 210916.   The fourth author was partially supported by NSF grants DMS \#1247354 and DMS \#1068946, and a Nebraska EPSCoR First Award. This collaboration began during visits supported by a travel grant from the AMS Mathematical Research Communities 2010 Commutative Algebra program.
Finally, much of the authors' collaborations took place at the University of Michigan, the University of Minnesota, and the Mathematical Sciences Research Institute; we thank these institutions for their hospitality.

\section{Basics of base $p$ expansions}

\begin{definition} 
\label{Expansion: D}  Given $\alpha \in (0,1]$, there exist unique integers $\digit{\alpha}{e}$ for every $e \geq 1$ such that 
$0 \leq \digit{\alpha}{e} \leq p-1$, $\alpha = \sum_{e \geq 1} \digit{\alpha}{e}\cdot p^{-e}$, and such that the integers $\digit{\alpha}{e}$ are not all eventually zero.  We call $\digit{\alpha}{e}$ the {$e^{\th}$ digit of $\alpha$} (base $p$), and we call the expression $\alpha = \sum_{e \geq 1} \digit{\alpha}{e} \cdot p^{-e}$ the {non-terminating expansion of $\alpha$} (base $p$). 
\end{definition}

\begin{definition}
Let $\alpha \in (0,1]$, and fix $e \geq 1$. We call $\tr{\alpha}{e}  :=  \digit{\alpha}{1} \cdot p^{-1} + \cdots + \digit{\alpha}{e} \cdot p^{-e}$ the {$e^{\th}$ truncation of $\alpha$} (base $p$).  We adopt the convention that $\tr{\alpha}{0} = 0$ and $\tr{\alpha}{\infty} = \alpha$.
\end{definition}

\begin{notation}
We adopt notation analogous to the standard decimal notation, using $``:"$ to distinguish between consecutive digits.  For example, we often write $\tr{\alpha}{e} = . \ \digit{\alpha}{1} : \digit{\alpha}{2} : \cdots : \digit{\alpha}{e} \ (\base p)$.
\end{notation}

\begin{convention}
\label{lpr: Conv}
Given a natural number $b > 0$ and an integer $m$, $\lpr{m}{b}$ denotes the least \emph{positive} residue of $m$ modulo $b$.  In particular, we have that $1 \leq \lpr{m}{b} \leq b$ for all $m \in \ZZ$.  
 Moreover, if $p$ and $b$ are relatively prime, $\order{p}{b} = \min \set{ k \geq 1: \lpr{p^k}{b} = 1}$, which we call the \emph{order of $p$ modulo $b$}.  In particular, note that $\order{p}{1} = 1$. 
\end{convention}

\begin{lemma}
\label{explicittruncation: L}
Fix $\lambda \in (0,1] \cap \Q$.  If we write $\lambda = \frac{a}{b}$, not necessarily in lowest terms, then \[ \digit{\lambda}{e} = \frac{ \lpr{ \paren{a}{p^{e-1}} }{b} \cdot p - \lpr{\paren{a}{p^e} }{b}}{b} \ \text{ and }  \
\tr{\lambda}{e} = \lambda - \frac{\lpr{\paren{a}{p^e}}{b}}{b p^e}.\] 
Note that it is important to keep in mind Convention \ref{lpr: Conv} when interpreting these identities.
\end{lemma}

\begin{proof}  Since $\digit{\lambda}{e} = p^e (\tr{\lambda}{e} - \tr{\lambda}{e-1})$, the first identity follows from the second. Setting $\delta = \lambda - \tr{\lambda}{e}$ and multiplying both sides of the equality $\frac{a}{b} = \lambda = \tr{\lambda}{e} + \delta$ by $b p^e$ shows that
\[ap^e = b p^e \tr{\lambda}{e} + b p^e \delta.\]  As $0 < \delta \leq p^{-e}$ and $p^e \tr{\lambda}{e} \in \NN$, it follows that $b p^e \delta$ is the least \emph{positive} residue of $a p^e$ modulo $b$.  Finally, substituting $\delta = \lambda - \tr{\lambda}{e}$ into $b p^e \delta = \lpr{\paren{a}{p^e}}{b}$ establishes the second identity.
\end{proof}

We gather some of the important basic properties of base $p$ expansions below.  

\begin{lemma}  
\label{BasicProperties: L}
Fix $\alpha$ and $\beta$ in $[0,1]$.
\begin{enumerate}
\item $\alpha \leq \beta$ if and only if $\tr{\alpha}{e} \leq \tr{\beta}{e}$ for all $e \geq 1$; if $\alpha < \beta$, then these inequalities are strict for $e \gg 0$.
\item If $(p^s-1) \cdot  \alpha \in \N$, then the base $p$ expansion of $\alpha$ is periodic, with period dividing $s$. In particular, if $\lambda = \frac{a}{b}$ with $p \nmid b$, then the base $p$ expansion of $\lambda$ is periodic with period equal to $\order{p}{b}$.
\item Suppose $\lambda = \frac{a}{b}$ with $p \nmid b$.  If $s= \order{p}{b}$, then for all $k \geq 1$, $p^{ks} \cdot \tr{\lambda}{ks} = (p^{ks} - 1) \cdot \lambda$. \end{enumerate}
\end{lemma}

\begin{proof} 
(1) follows by definition; (2) follows immediately from Lemma \ref{explicittruncation: L}; (3) follows from (2).
\end{proof}

\begin{lemma}
\label{Difference: L}  Consider $\alpha < \beta$ in $(0,1]$, and set $\Delta_e := p^e \tr{\beta}{e} - p^{e} \tr{\alpha}{e}$.  Note that, by Lemma \ref{BasicProperties: L}, the integer $\ell= \min \set{ e : \Delta_e \geq 1}$ is well-defined.  Moreover, the following hold:
\begin{enumerate}
 \item   The sequence $\set{ \Delta_e }_{e \geq 1}$ is non-negative, non-decreasing, and unbounded.
\item Suppose $\beta = \frac{a}{b}$ with $p \nmid b$.  If $s = \order{p}{b}$, then $\Delta_{\ell + s + k} \geq p^k + 1$ for every $k \geq 0$.
\end{enumerate}
\end{lemma}

\begin{proof} 
We first observe that the following recursion holds.
\begin{equation} 
\label{recursion: e}
\Delta_{e+1} =  p  \cdot \Delta_e + \digit{\beta}{e+1} - \digit{\alpha}{e+1} \text{ for every } e \geq 0.
\end{equation}

Setting $e = \ell$ in \eqref{recursion: e} and noting that $\Delta_\ell \geq 1$ shows that 
\begin{align*}
\Delta_{\ell+1}  = p \cdot \Delta_{\ell} + \digit{\beta}{\ell+1} - \digit{\alpha}{\ell+1} & = (p-1) \cdot \Delta_{\ell} + \Delta_{\ell} + \digit{\beta}{\ell+1} - \digit{\alpha}{\ell+1} \\ & \geq (p-1) \cdot 1 + \Delta_{\ell} + \digit{\beta}{\ell+1} - \digit{\alpha}{\ell+1} \\ 
 & \geq \Delta_{\ell} + \digit{\beta}{\ell+1}.
\end{align*}
Furthermore, an induction on $e \geq \ell$ shows that
\begin{equation}
\label{unbounded: e}
\Delta_{e+1} \geq \Delta_e + \digit{\beta}{e+1} \text{ for every $e \geq \ell$.}
\end{equation}
Thus, $\set{ \Delta_e }_{e \geq 1}$ is non-decreasing, and as we consider non-terminating expansions, $\digit{\beta}{e} \neq 0$ for infinitely many $e$, so that \eqref{unbounded: e} also shows that $\Delta_{e+1} > \Delta_e$ for infinitely many $e$. We conclude that $\set{ \Delta_e}_{e \geq 1}$ is unbounded, and it remains to establish (2). 

By definition,  $\digit{\beta}{\ell} - \digit{\alpha}{\ell} = \Delta_{\ell} \geq 1$, and hence $\digit{\beta}{\ell} \geq 1$.  In fact, setting $s = \order{p}{b}$, Lemma \ref{BasicProperties: L} states that $\digit{\beta}{\ell + s} = \digit{\beta}{\ell} \geq 1$, and applying \eqref{unbounded: e} with $e=\ell + s-1$  then shows that
\[ \Delta_{\ell + s} \geq \Delta_{\ell + s -1 } + \digit{\beta}{\ell + s} \geq 2.\] 
Hence, (2) holds for $k=0$.  Utilizing  \eqref{recursion: e}, an induction on $k$ completes the proof.  
\end{proof}

\section{$F$-pure thresholds of \qh polynomials:  A discussion}
\label{fptDiscussion: Section}

We adopt the following convention from this point onward.

\begin{convention} \label{generalConvention: Con}
Throughout this article, $\LL$ will denote a field of characteristic $p>0$, and $\m$ will denote the ideal generated by the variables in $R=\LL[x_1, \cdots, x_n]$.  
\end{convention}

\begin{definition}
\label{fptDef: D}
Consider a polynomial $f \in \m$, and for every $e \geq 1$, set  \[ \new{e} = \max \set{ N : f^N \notin \bracket{\m}{e}}.\] 
An important property of these integers is that $\set{ p^{-e} \cdot \new{e}}_{e \geq 1}$ is a non-decreasing sequence contained in the open unit interval \cite{MTW2005}.  Consequently, the limit  
\[ \fpt{f} : = \lim_{e \to \infty} \frac{\new{e}}{p^e}  \in (0,1] \] exists, and is called the \emph{$F$-pure threshold of $f$}.
\end{definition}

The following illustrates important properties of $F$-pure thresholds;  we refer the reader to \cite[Proposition 1.9]{MTW2005} or \cite[Key Lemma 3.1]{Singularities} for a proof of the first, and \cite[Corollary 4.1]{Singularities} for a proof of the second.

\begin{proposition}
\label{Truncation: P}
Consider a polynomial $f$ contained in $\m$.
\begin{enumerate}  
\item The base $p$ expansion of the $F$-pure threshold determines $\set{ \new{e}}_{e \geq 1}$; more precisely,   
\[\new{e} = p^e \cdot \tr{\fpt{f}}{e} \text{ for every $e \geq 1$}. \]   
\item The $F$-pure threshold is bounded above by the rational numbers determined by its trailing digits (base $p$);  more precisely, $\fpt{f}$ is less than or equal to \[ .\ \digit{\fpt{f}}{s} : \digit{\fpt{f}}{s+1} : \cdots : \digit{\fpt{f}}{s+k} : \cdots \ (\base p) \text{ for every $s \geq 1$}.\]  
\end{enumerate}
\end{proposition}
 
\subsection{A discussion of the main results}  In this subsection, we gather the main results of this article.  Note that the proofs of these results appear in Section \ref{details: S}.

\begin{convention}
Given a polynomial $f$, we use $\jac{f}$ to denote the ideal of $R$ generated by the partial derivatives of $f$.  If $f$ is homogeneous under some $\NN$-grading on $R$,  each partial derivative $\df{i}$ of $f$ is also \qhns, and { if $\df{i}\neq 0$}, then $\qdeg \df{i} = \qdeg f - \qdeg x_i$.  Furthermore, if $p \nmid \qdeg f$, then Euler's relation \[ \qdeg f \cdot f = \sum \qdeg x_i \cdot x_i \cdot \df{i} \] shows that $f \in \jac{f}$.  Thus, if $p \nmid \qdeg(f)$ and $\LL$ is perfect, the Jacobian criterion states that $\sqrt{\jac{f}} = \m$ if and only if $f$ has an isolated singularity at the origin.
\end{convention} 

\begin{theorem}
\label{Main: T}
Fix an $\NN$-grading on $\LL[x_1, \cdots, x_n]$.  Consider a \qh polynomial $f$ with $\sqrt{\jac{f}} = \m$, and write  $\lambda:=\min \set{ \frac{ \sum \qdeg x_i}{\qdeg f}, 1} = \frac{a}{b}$ in lowest terms.
\begin{enumerate}
\item \label{MT1: p} If $\fpt{f} \neq \lambda$, then 
\begin{align*}  
\fpt{f} &= \lambda  - \( \frac{ \lpr{\paren{a}{p^L}}{b} + b E }{b p^L}\)\\ 
&= \tr{ \lambda }{L} - \frac{E}{p^L} 
\end{align*}  
for some pair $(L,E) \in \NN^2$ with $L \geq 1$ and $0 \leq E \leq n-1-\up{\frac{\lpr{\paren{a}{p^L}}{b} + a}{b}}$.
\item \label{MT2: p} If $p > (n-2) \cdot b$ and $p \nmid b$, then $1 \leq L \leq \order{p}{b}$; note that $\order{p}{1} = 1$.
\item \label{MT3: p}
If $p > (n-2) \cdot b$ and $p > b$,  then $ a < \lpr{ \paren{a}{p^e} }{b} \text{ for all } 1 \leq e \leq L-1.$
\item \label{MT4: p} If $p> (n-1) \cdot b$, then there exists a unique pair $(L,E)$ satisfying the conclusions of \eqref{MT1: p}.
\end{enumerate}
\end{theorem}
{{
We postpone the proof of Theorem  \ref{Main: T} to  Subsection \ref{mainProofs:  SS}.}}
The remainder of this subsection is focused on parsing the statement of Theorem \ref{Main: T}, and presenting some related results.  The reader interested in seeing examples should consult Section \ref{fptExamples: Section}.

\begin{remark}[Two points of view]
Each of the two descriptions of $\fpt{f}$ in Theorem \ref{Main: T}, which are equivalent by Lemma \ref{explicittruncation: L}, are useful in their own right. For example, the first description plays a key role in Section \ref{fptExamples: Section}.  
On the other hand, the second description makes it clear that either $\fpt{f} = \lambda$, or $\fpt{f}$ is a rational number whose denominator is a power of $p$,  and further, describes how ``far" $\fpt{f}$ is from being a truncation of $\lambda$;  these observations allow us to address the questions of Schwede and of the first author noted in the introduction.
\end{remark}

The second point of Theorem \ref{Main: T} also immediately gives a bound on the power of $p$ appearing in the denominator of $\fpt{f}$ whenever $\fpt{f} \neq \lambda$ and $p \gg 0$.  For emphasis, we record this bound in the following corollary.

\begin{corollary}
In the context of Theorem \ref{Main: T}, if $\fpt{f} \neq \lambda$, and both $p > (n-2) \cdot b$ and $p \nmid b$, then $p^{\order{p}{b}} \cdot \fpt{f} \in \NN$.  In particular, for all such primes, $p^{\phi(b)}  \cdot \fpt{f} \in \NN$, where $\phi$ denotes Euler's phi function.
\end{corollary}

Using the techniques of the proof of Theorem \ref{Main: T}, we can analogously find a bound for the power of $p$ appearing in the denominator of $\fpt{f}$ whenever $\fpt{f} \neq \lambda$ and $p$ is not large, which we record here. 

\begin{corollary}
\label{UniformLBounds: C}
In the setting of Theorem \ref{Main: T}, if $\fpt{f} \neq \lambda$ and $p \nmid b$, then $p^{M} \cdot \fpt{f} \in \NN$, where $M :=2 \cdot \phi(b) + \up{\log_2(n-1)},$ and $\phi$ denotes Euler's phi function.
\end{corollary}

\begin{remark}
We emphasize that the constant $M$ in Corollary \ref{UniformLBounds: C} depends only on the number of variables $n$ and the quotient $\frac{\sum \qdeg(x_i)}{\qdeg f}  = \frac{a}{b}$, but not on the particular values of $\qdeg x_i$ and $\qdeg f$;  this subtle point will play a key role in Proposition \ref{ACC: P}.
\end{remark}

\begin{remark}[Towards minimal lists] 
For $p \gg 0$, the bounds for $L$ and $E$ appearing in Theorem \ref{Main: T} allows one to produce a finite list of possible values of $\fpt{f}$ for each class of $p$ modulo $\qdeg f$.  We refer the reader to Section \ref{fptExamples: Section} for details and examples.
\end{remark}

The uniqueness statement in point \eqref{MT4: p} of the Theorem \ref{Main: T}  need not hold in general.

\begin{example}[Non-uniqueness in low characteristic]
If $p=2$ and $f \in \LL[x_1,x_2,x_3]$ is any $\LL^{\ast}$-linear combination of $x_1^7, x_2^7, x_3^7$, then $f$ satisfies the hypotheses of Theorem \ref{Main: T}, under the standard grading.  Using \cite{Diagonals}, one can directly compute that $\fpt{f} = \frac{1}{4}$.  On the other hand, the identities 
\begin{align*}
\frac{1}{4}& =\frac{3}{7} - \( \frac{\lpr{\parenone{3\cdot 2^2}} {7} + 7 \cdot 0}{7 \cdot 2^2} \) =\tr{\frac{3}{7}}{2}\\
&  = \frac{3}{7} - \( \frac{\lpr{\parenone{3 \cdot 2^3} }{7} + 7 \cdot 1}{7 \cdot 2^3} \) =\tr{\frac{3}{7}}{3}-\frac{1}{2^3}
\end{align*}
show that the pairs $(L,E) = (2,0)$ and $(L,E) = (3,1)$ both satisfy the conclusions in Theorem \ref{Main: T}.  We point out that the proof of Theorem \ref{Main: T}, being somewhat constructive, predicts the choice of $(L,E)=(2,0)$, but does not ``detect'' the choice of $(L,E) = (3,1)$.
\end{example}

Before concluding this section, we present the following related result;  like Theorem \ref{Main: T} and Corollary \ref{UniformLBounds: C},  its proof relies heavily on Proposition \ref{nubounds: P}.  However, in contrast to these results, its focus is on showing that $\fpt{f} = \min \set{ \(\sum \qdeg x_i\)/ \qdeg f, 1}$ for $p \gg 0$ in a very specific setting, as opposed to describing $\fpt{f}$ when it differs from this value.

\begin{theorem}
\label{Secondary: T}
In the context of Theorem \ref{Main: T}, suppose that $\sum \qdeg x_i > \qdeg f$, so that $\rho:= \frac{\sum \qdeg x_i}{\qdeg f}$ is greater than $1$.  If $p> \frac{n-3}{\rho - 1}$, then $\fpt{f} = 1$.
\end{theorem}

As we see below, Theorem \ref{Secondary: T} need not hold in low characteristic.

\begin{example}[Illustrating the necessity of $p \gg 0$ in Theorem \ref{Secondary: T}] Set $f=x_1^d + \cdots + x_n^d$.  If $n>d>p$, then $f \in \bracket{\m}{}$, and hence $f^{p^{e-1}} \in \bracket{\m}{e}$ for all $e \geq 1$.  Consequently, $\new{e} \leq p^{e-1}-1$, and therefore $\fpt{f} = \lim \limits_{e \to \infty} p^{-e} \cdot \new{e} \leq p^{-1}$.
\end{example}

\section{$F$-pure thresholds of \qh polynomials:  Examples}
\label{fptExamples: Section}

In this section, we illustrate, via examples, how Theorem \ref{Main: T} may be used to produce ``short," or even minimal, lists of possible values for $F$-pure thresholds.
We begin with the most transparent case:  If $\qdeg f = \sum \qdeg x_i$, then the statements in Theorem \ref{Main: T} become less technical.  Indeed, in this case, $a=b=1$, and hence $\order{p}{b} = 1 = \lpr{m}{b}$ for every $m \in \NN$.  In this context, substituting these values into Theorem \ref{Main: T} recovers the following identity, originally discovered by Bhatt and Singh under the standard grading.

\begin{example} \cite[Theorem 1.1]{BhattSingh}
\label{calabiYau: E} 
Fix an $\NN$-grading on $\LL[x_1, \cdots, x_n]$.  Consider a homogeneous polynomial $f$ with $d:=\qdeg f =\sum \qdeg x_i$ and $\sqrt{\jac{f}} = \m$.  If $p>n-2$ and $\fpt{f} \neq 1$, then \[ \fpt{f} = 1 -  A \cdot p^{-1} \text{ for some integer $1 \leq A \leq d-2$}.\] \end{example}

Next, we consider the situation when $\qdeg f = \sum \qdeg x_i + 1$;  already, we see that this minor modification leads to a more complex statement.

\begin{corollary}  
\label{aCY: C}
Fix an $\NN$-grading on $\LL[x_1, \cdots, x_n]$.  Consider a \qh polynomial $f$ with $d:=\qdeg f =\sum \qdeg x_i + 1$ and $\sqrt{\jac{f}} = \m$, and suppose that $p > (n-2) \cdot d$.
\begin{enumerate}
\item If $\fpt{f} = 1 - \frac{1}{d}$, then $p \equiv 1 \bmod d$.
\item If $\fpt{f} \neq 1-\frac{1}{d}$, then $\fpt{f} = 1 - \frac{1}{d} -  \( A- \frac{\lpr{p}{d}}{d} \) \cdot p^{-1}$ for some integer $A$ satisfying
 \begin{enumerate}
\item $1 \leq A \leq d-2$ if $p \equiv -1 \bmod d$, and
\item $1 \leq A \leq d-3$ otherwise.
\end{enumerate}
\end{enumerate}
\end{corollary}

\begin{proof}  We begin with $(1)$:  Lemma \ref{explicittruncation: L} implies that $\digit{\( d^{-1} \)}{1} \leq \digit{ \( d^{-1} \)}{s} \text{ for $s \geq 1$}$, and hence that 
\begin{equation} 
\label{almostcy1: e}
\digit{\(1 - d^{-1}\)}{1} = p-1- \digit{\(d^{-1}\)}{1} \geq p-1- \digit{\(d^{-1}\)}{s} =   \digit{\(1 - d^{-1}\)}{s}\end{equation}  
for every $s \geq 1$.  However, if $\fpt{f} = 1 - d^{-1}$,  Proposition \ref{Truncation: P} implies that $\digit{\( 1 - d^{-1} \)}{1} \leq \digit{\( 1 - d^{-1} \)}{s}$ for every $ s \geq 1$.  Consequently, equality holds throughout \eqref{almostcy1: e}, and hence $\digit{\( d^{-1} \)}{1} = \digit{ \( d^{-1} \)}{s}$ for every $s \geq 1$, which by Lemma \ref{explicittruncation: L} occurs if and only if $p \equiv 1 \bmod d$.

We now address the second point:  In this setting, Theorem \ref{Main: T} states that $\fpt{f} \in p^{-L} \cdot \NN$ for some integer $L \geq 1$.  We will now show that $L$ must equal one:  Indeed, otherwise $L \geq 2$, which allows us to set $e = 1$ in the third point Theorem \ref{Main: T} to deduce that
\[ 1 \leq d-1 < \lpr{ \parenone{p(d-1)} }{d} = d - \lpr{p}{d},\] and hence that 
$\lpr{p}{d} < 1$, which is impossible, as $\lpr{p}{d}$ is always a positive integer.  We conclude that $L=1$, and the reader may verify that substituting $L=1$, $\lpr{\parenone{p(d-1)}}{d} = d - \lpr{p}{d}$, and $A:=E+1$ into Theorem \ref{Main: T} produces the desired description of $\fpt{f}$.
\end{proof}

\subsection{The two variable case}  We now shift our focus to the two variable case of Theorem \ref{Main: T}, motivated by the following example.

\begin{example}
\label{HaraMonsky: E}
In \cite[Corollary 3.9]{Hara2006}, Hara and Monsky independently described the possible values of $\fpt{f}$ whenever $f$ is homogeneous in two variables (under the standard grading) of degree $5$ with an isolated singularity at the origin over an algebraically closed field (and hence, a product of five distinct linear forms), and $p \neq 5$;  we recall their computation below (the description in terms of truncations is our own).  

\begin{itemize}
\item If $p \equiv 1 \bmod 5$, then $\fpt{f} = \frac{2}{5}$ or $\frac{2p-2}{5p} =\tr{\frac{2}{5}}{1}$. 
\item If $p \equiv 2 \bmod 5$, then $\fpt{f} = \frac{2p^2-3}{5 p^2} =\tr{\frac{2}{5}}{2}$  or $\frac{2p^3-1}{5p^3} =\tr{\frac{2}{5}}{3}$.
\item If $p \equiv 3 \bmod 5$, then $\fpt{f} = \frac{2p-1}{5p} =\tr{\frac{2}{5}}{1}$.
\item If $p \equiv 4 \bmod 5$, then $\fpt{f} = \frac{2}{5}$ or $\frac{2p-3}{5p} =\tr{\frac{2}{5}}{1}$ or $\frac{2p^2-2}{5 p^2} =\tr{\frac{2}{5}}{2}$.
\end{itemize}

The methods used in \cite{Hara2006} rely on so-called ``syzygy gap" techniques and the geometry of $\mathbb{P}^1$, and hence differ greatly from ours.  In this example, we observe the following:
First, the $F$-pure threshold is always $\lambda = \frac{2}{5}$, or a truncation of $\frac{2}{5}$.  Secondly, there seem to be fewer choices for the truncation point $L$ than one might expect, given Theorem \ref{Main: T}. 
\end{example}

In this subsection, we show that the two observations from Example \ref{HaraMonsky: E}
hold in general in the two variable setting.  We now work in the context of Theorem \ref{Main: T} with $n=2$, and relabel the variables so that $f \in \LL[x,y]$.  Note that if $\qdeg f < \qdeg xy$, then $\fpt{f} = 1$, by Theorem \ref{Secondary: T} (an alternate justification:  this inequality is satisfied if and only if, after possibly re-ordering the variables, $f = x + y^m$ for some $m \geq 1$, in which case one can directly compute that $\new{e} = p^e-1$, and hence that $\fpt{f} = 1$).  Thus, the interesting case here is when $\qdeg f \geq \qdeg xy$.  In this case, one obtains the following result.

\begin{theorem}[cf. Theorem \ref{Main: T}]
\label{Main2D: T}
Fix an $\NN$-grading on $\LL[x,y]$.  Consider a \qh polynomial $f$ with $\sqrt{\jac{f}} = \m$ and $\qdeg f \geq \qdeg xy$.  If $\fpt{f} \neq \frac{\qdeg xy}{\qdeg f} = \frac{a}{b}$, written in lowest terms, then
\[ \fpt{f} = \tr{ \frac{ \qdeg xy}{\qdeg f}}{L} = \frac{ \qdeg xy}{\qdeg f} - \frac{\lpr{\paren{a}{p^L}}{b}}{b \cdot p^L} \] for some integer $L$ satisfying the following properties:
\begin{enumerate}
\item If $p \nmid b$, then $1 \leq L \leq \order{p}{b}$.
\item If $p > b$, then $a < \lpr{\paren{a}{p^e}}{b}$ for all $1 \leq e \leq L-1$.
\item $1 \leq \lpr{\paren{a}{p^L}}{b} \leq b-a$ for \emph{all} possible values of $p$.
\end{enumerate}
\end{theorem}

\begin{proof}
Assuming $\fpt{f} \neq \frac{\qdeg xy}{\qdeg f}$, the bounds for $E$ in Theorem \ref{Main: T} become \[    0 \leq E \leq 1 - \up{ \frac{\lpr{\paren{a}{p^L}}{b} + a}{b} }. \]   As the rounded term above is always either one or two, the inequality forces it to equal one, so that $E=0$, which shows that $\fpt{f}$ is a truncation of $\frac{\qdeg xy}{\qdeg f}$.  Moreover, the fact that the rounded term above equals one also implies that $\lpr{\paren{a}{p^L}}{b} + a \leq b$.
\end{proof}

\begin{remark}
Though the first two points in Theorem \ref{Main2D: T} appear in Theorem \ref{Main: T}, the third condition is special to the setting of two variables.  Indeed, this extra condition will be key in eliminating potential candidate $F$-pure thresholds.  For example, this extra condition allows us to recover the data in Example \ref{HaraMonsky: E}.  Rather than justify this claim, we present two new examples. 
\end{remark}

\begin{example}
\label{6LinearFactors: E}
Let $f \in \LL[x,y]$ be as in Theorem \ref{Main2D: T}, with $\frac{ \qdeg(xy) }{\qdeg f} = \frac{1}{3}$.  For $p\geq 5$, the following hold: 
\begin{itemize}
\item If $p \equiv 1 \bmod 3$, then $\fpt{f} = \frac{1}{3}$ or $\tr{\frac{1}{3}}{1} = \frac{1}{3} - \frac{1}{3p}$.
\item If $p \equiv 2 \bmod 3$, then $\fpt{f} = \frac{1}{3}$ or $\tr{\frac{1}{3}}{1} = \frac{1}{3} - \frac{2}{3p}$ or $\tr{\frac{1}{3}}{2} = \frac{1}{3} - \frac{1}{3p^2}$.
\end{itemize}
\end{example}

In Example \ref{6LinearFactors: E}, the second and third points of Theorem \ref{Main2D: T} were uninteresting, as they did not ``whittle away" any of the candidate $F$-pure thresholds identified by the first point of Theorem \ref{Main2D: T}.  The following example is more interesting, as we will see that both of the second and third points of Theorem \ref{Main2D: T}, along with Proposition \ref{Truncation: P}, will be used to eliminate potential candidates.

\begin{example}
\label{7LinearFactors: E}
Let $f \in \LL[x,y]$ be as in Theorem \ref{Main2D: T}, with $\frac{ \qdeg(xy) }{\qdeg f} = \frac{2}{7}$. For $p\geq 11$, the following hold:
\begin{itemize}
\item If $p \equiv 1 \bmod 7$, then $\fpt{f} = \frac{2}{7}$ or $\tr{\frac{2}{7}}{1} = \frac{2}{7} - \frac{2}{7p}$.
\item If $p \equiv 2 \bmod 7$, then $\fpt{f} = \tr{\frac{2}{7}}{1} = \frac{2}{7} - \frac{4}{7p}$ or $\tr{\frac{2}{7}}{2} = \frac{2}{7} - \frac{1}{7p^2}$.
\item If $p \equiv 3 \bmod 7$, then $\fpt{f} = \tr{\frac{2}{7}}{2} = \frac{2}{7} - \frac{4}{7p^2}$ or $\tr{\frac{2}{7}}{3} = \frac{2}{7} - \frac{5}{7p^3}$ or $\tr{\frac{2}{7}}{4} = \frac{2}{7} - \frac{1}{7p^4}$.
\item If $p \equiv 4 \bmod 7$, then $\fpt{f} = \tr{\frac{2}{7}}{1} = \frac{2}{7} - \frac{1}{7p}$.
\item If $p \equiv 5 \bmod 7$, then $\fpt{f} = \tr{\frac{2}{7}}{1} = \frac{2}{7} - \frac{3}{7p}$ or $\tr{\frac{2}{7}}{2} = \frac{2}{7} - \frac{1}{7p^2}$.
\item If $p \equiv 6 \bmod 7$, then $\fpt{f} = \frac{2}{7}$ or $\tr{\frac{2}{7}}{1} = \frac{2}{7} - \frac{5}{7p}$ or $\tr{\frac{2}{7}}{2}  = \frac{2}{7} - \frac{2}{7p^2}$.
\end{itemize}
\noindent For the sake of brevity, we only indicate how to deduce the lists for $p \equiv 3 \bmod 7$ and $p \equiv 4 \bmod 7$.  Similar methods can be used for the remaining cases.

$(p \equiv 3 \bmod 7)$. \ In this case, it follows from Lemma \ref{explicittruncation: L} that $\digit{\( \frac{2}{7} \)}{1} = \frac{2p-6}{7}$ and $\digit{\( \frac{2}{7} \)}{5} = \frac{p-3}{7}$.  In light of this, the second point of Proposition \ref{Truncation: P}, which shows that the first digit of $\fpt{f}$ must be the smallest digit, implies that $\fpt{f} \neq \frac{2}{7}$.  Thus, the first point of Theorem \ref{Main2D: T} states that 
\[ \fpt{f} = \tr{\frac{2}{7}}{L} \text{ for some $1 \leq L \leq \order{p}{7} = 6$, as $p \equiv 3 \bmod 7$}. \]
However, as $2 \not \leq \lpr{\paren{2}{p^4}}{7} = 1$, the second point of Theorem \ref{Main2D: T} eliminates the possibilities that $L=5$ or $6$.  Moreover, as $\lpr{\paren{2}{p}}{7} = 6 \not \leq 7-2 = 5$, the third point of Theorem \ref{Main2D: T} eliminates the possibility that $L=1$.  Thus, the only remaining possibilities are $\fpt{f} = \tr{\frac{2}{7}}{2}, \tr{\frac{2}{7}}{3}$, and $\tr{\frac{2}{7}}{4}$.

$(p \equiv 4 \bmod 7)$. \  As before, we compute that $\digit{\(\frac{2}{7}\)}{1} = \frac{2p-1}{7}$ is greater than $\digit{\(\frac{2}{7}\)}{2} = \frac{p-4}{7}$, and hence it again follows the second point of Proposition \ref{Truncation: P} that $\fpt{f} \neq \frac{2}{7}$.  Consequently, the first point of Theorem \ref{Main2D: T} states that \[ \fpt{f} = \tr{\frac{2}{7}}{L} \text{ for some $1 \leq L \leq \order{p}{7} = 3$, as $p \equiv 4 \bmod 7$.}\]
However, we observe that $2 \not \leq \lpr{\paren{2}{p^2}}{7} = 1$, and hence the second point of Theorem \ref{Main2D: T} eliminates the possibility that $L=2$ or $3$.    Thus, the only remaining option is that $\fpt{f} = \tr{\frac{2}{7}}{1}$.
\end{example}

\begin{remark}[Minimal lists]
In many cases, we are able to verify that the ``whittled down" list obtained through the application of Theorems \ref{Main2D: T} and \ref{Main: T} and Proposition \ref{Truncation: P} is, in fact, minimal.  For example, every candidate listed in Example \ref{HaraMonsky: E} is of the form $\fpt{f}$, where $f$ varies among the polynomials $x^5 + y^5$, $x^5+x y^4$, and $x^5 + x y^4 + 7 x^2 y^3$, and $p$ various among the primes less than or equal to $29$.
\end{remark}

An extreme example of the ``minimality" of the lists of candidate thresholds  appears below.  Note that, in this example, the list of candidate thresholds is so small that it actually determines the precise value of $\fpt{f}$ for $p \gg 0$.

\begin{example}[$F$-pure thresholds are precisely determined] \label{determined: E}  Let $f \in \LL[x,y]$ be as in Theorem \ref{Main2D: T}, with $\frac{ \qdeg(xy) }{\qdeg f} = \frac{3}{5}$;  for example, we may take $f = x^5 + x^3 y +  x y^2$, under the grading given by $(\qdeg x, \qdeg y) = (1,2)$.  Using Theorem \ref{Main2D: T} and Proposition \ref{Truncation: P} in a manner analogous to that used in Example \ref{7LinearFactors: E}, we obtain the following complete description of $\fpt{f}$ for $p\geq7$.
\begin{itemize}
\item If $p \equiv 1 \bmod 5$, then $\fpt{f} = \frac{3}{5}$.
\item If $p \equiv 2 \bmod 5$, then $\fpt{f} = \tr{\frac{3}{5}}{1} = \frac{3}{5} - \frac{1}{5p}.$
\item If $p \equiv 3 \bmod 5$, then $\fpt{f} = \tr{\frac{3}{5}}{2} = \frac{3}{5} - \frac{2}{5p^2}.$
\item If $p \equiv 4 \bmod 5$, then $\fpt{f} = \tr{\frac{3}{5}}{1} = \frac{3}{5} - \frac{2p}{5}.$
\end{itemize}
\end{example}

We conclude this section with one final example illustrating ``minimality".  In this instance, however, we focus on the higher dimensional case.  Although the candidate list for $F$-pure thresholds produced by Theorem \ref{Main: T} is more complicated (due to the possibility of having a non-zero ``$E$" term when $n > 2$), the following example shows that we can nonetheless obtain minimal lists in these cases using methods analogous to those used in this section's previous examples. 

\begin{example}[Minimal lists for $n\geq3$]
Let $f \in \LL[x,y,z]$ satisfy the hypotheses of Theorem \ref{Main: T}, with $\frac{\qdeg xyz}{\qdeg f} = \frac{2}{3}$.  Using the bounds for $E$ and $L$ therein, we obtain the following for $p \geq 5$:
\begin{itemize}
\item If $p \equiv 1 \bmod 3$, then $\fpt{f} = \frac{2}{3}$ or $\tr{\frac{2}{3}}{1} = \frac{2}{3} - \frac{2}{3p}$.
\item If $p \equiv 2 \bmod 3$, then $\fpt{f} = \tr{\frac{2}{3}}{1} = \frac{2}{3} - \frac{1}{3p}$ or $\tr{\frac{2}{3}}{1} - \frac{1}{p} = \frac{2}{3} - \frac{4}{3p}$
\end{itemize}  We claim that this list is minimal.  In fact, if $f = x^9 + x y^4 + z^3$, homogeneous under the grading determined by $(\qdeg x, \qdeg y, \qdeg z) = (1,2,3)$, we obtain each of these possibilities as $p$ varies.
\end{example}

\section{$F$-pure thresholds of \qh polynomials:  Details} \label{details: S}

Here, we prove the statements referred to in Section \ref{fptDiscussion: Section}; we begin with some preliminary results.  

\subsection{Bounding the defining terms of the $F$-pure threshold}
\label{prelimDetails: ss}

This subsection is dedicated to deriving bounds for $\new{e}$.  Our methods for deriving lower bounds are an extension of those employed by Bhatt and Singh in \cite{BhattSingh}.

\begin{lemma}
\label{generalfptBound: L}
If $f\in \LL[x_1, \cdots, x_n]$ is \qh under some $\NN$-grading, then  for every $e \geq 1$, $\new{e} \leq \down{ (p^e-1) \cdot \frac{\sum \qdeg x_i}{\qdeg f} }$.  In particular, $\fpt{f} \leq \min \set{\frac{\sum \qdeg x_i}{\qdeg f}, 1}$. 
\end{lemma}
 
\begin{proof}
By Definition \ref{fptDef: D}, it suffices to establish the upper bound on $\new{e}$.  However, as $f^{\new{e}} \notin \bracket{\m}{e}$, there is a supporting monomial $\mu = x_1^{a_1} \cdots x_n^{a_n}$ of $f^{\new{e}}$ not in $\bracket{\m}{e}$, and comparing degrees shows that
\[ \new{e} \cdot \qdeg f  = \qdeg \mu = \sum a_i \cdot \qdeg x_i \leq \( p^e-1\) \cdot \sum \qdeg x_i.\]
\end{proof}

\begin{corollary}
\label{boundsOnFirstDiff: C}
Let $f\in \LL[x_1, \cdots, x_n]$ be a \qh polynomial under some $\NN$-grading, and write $\lambda = \min \set{\frac{\sum \qdeg x_i}{\qdeg f}, 1}=\frac{a}{b}$ in lowest terms.  If $\fpt{f} \neq \lambda$, then $\Delta_e:=p^e \tr{\lambda}{e} - p^e \tr{\fpt{f}}{e}$ defines a non-negative, non-decreasing, unbounded sequence. 
Moreover, if $p \nmid b$, then \[1 \leq \min \set{ e : \Delta_e \neq 0 } \leq \order{p}{b}.\] 
\end{corollary}
 
\begin{proof}  By Lemma \ref{generalfptBound: L}, the assumption that $\fpt{f} \neq \lambda$ implies that $\fpt{f} < \lambda$, the so the asserted properties of $\set{ \Delta_e}_e$ follow from Lemma \ref{Difference: L}. Setting $s:= \order{p}{b}$,  it follows from Lemma \ref{explicittruncation: L} that 
\[ \lambda:= . \ \overline{\digit{\lambda}{1} : \cdots : \digit{\lambda}{s} } \ (\base p).  \]

By means of contradiction, suppose $\Delta_{s} = 0$, so that  $\tr{\lambda}{s} = \tr{\fpt{f}}{s}$, i.e., so that
\begin{equation}
\label{fptexp: e}
 \fpt{f} = . \ \digit{\lambda}{1} : \cdots : \digit{\lambda}{s} : \digit{\fpt{f}}{s+1} : \digit{\fpt{f}}{s+2} : \cdots \ (\base p).\end{equation}
As $\fpt{f} \leq \lambda$, comparing the tails of the expansions of $\fpt{f}$ and $\lambda$ shows that 
\[. \ \digit{\fpt{f}}{s+1}: \cdots : \digit{\fpt{f}}{2s} \ (\base p) \leq . \ \digit{\lambda}{1} : \cdots : \digit{\lambda}{s} \ (\base p).\] 
On the other hand, comparing the first $s$ digits appearing in the second point of Proposition \ref{Truncation: P}, recalling the expansion \eqref{fptexp: e}, shows that 
\[ . \ \digit{\lambda}{1} : \cdots : \digit{\lambda}{s} \ (\base p) \leq . \ \digit{\fpt{f}}{s+1} : \cdots \digit{\fpt{f}}{2s} \ (\base p),
\]  and thus we conclude that $\digit{\fpt{f}}{s+e} = \digit{\lambda}{s+e}$ for every $1 \leq e \leq s$, i.e., that $\Delta_{2s} = 0$.  Finally, a repeated application of this argument will show that $\Delta_{ms} = 0$ for every $m \geq 1$, which implies that $\fpt{f} = \lambda$, a contradiction.
\end{proof}

\begin{notation}
If $R$ is any $\NN$-graded ring, and $M$ is a graded $R$-module, $\Deg{M}{d}$   will denote the degree $d$ component of $M$, and $\Deg{M}{\leq d}$ and $\Deg{M}{\geq d}$ the obvious $\Deg{R}{0}$ submodules of $M$.  Furthermore, we use $H_M(t) : = \sum_{d \geq 0} \dim \Deg{M}{d} \cdot t^d$ to denote the Hilbert series of $M$.
\end{notation}

For the remainder of this subsection, we work in the following context.

\begin{setup} 
\label{isoSing: setup}
Fix an $\NN$-grading on $R=\LL[x_1, \cdots, x_n]$, and consider a \qh polynomial $f \in \m$ with $\sqrt{\jac{f}} = \m$.  In this context, $\df{1}, \cdots, \df{n}$ form a \qh system of parameters for $R$, and hence a regular sequence.  Consequently, if we set $J_k = \( \df{1},\cdots, \df{k} \)$,  the sequences  \[ 0 \to \(R /  J_{k-1}\)( -\qdeg f + \qdeg x_k ) \stackrel{ \df{k}}{\longrightarrow} R /  J_{k-1} \to  R / J_k \to 0  \]  are exact for every $1 \leq k \leq n$.  Furthermore,  using the fact that the Hilbert series is additive across short exact sequences, the well-known identities $H_R(t)  = \prod_{i=1}^n \frac{1}{1-t^{\qdeg x_i}} \text{ and } H_{M(-s)}(t) = t^s  H_M(t)$ imply that
\begin{equation}
\label{HilbertSeries: e}
H_{R/\jac{f}}(t) = \prod \limits_{i=1}^n \frac{ 1 - t^{\qdeg f-\qdeg x_i} }{ 1 - t^{\qdeg x_i} }, 
\end{equation}
an identity that will play a key role in what follows.

\end{setup}

\begin{lemma} 
\label{Anurag'sTrick: L} Under Setup \ref{isoSing: setup}, we have that $\Colon{\bracket{\m}{e}}{\jac{f}} \setminus \bracket{\m}{e} \subseteq \Deg{R}{\geq (p^e + 1) \cdot \sum \qdeg x_i - n \cdot \qdeg f}$.
\end{lemma}

\begin{proof} To simplify notation, set $J = \jac{f}$.  By \eqref{HilbertSeries: e},  the degree of $H_{R/J}(t)$ (a polynomial, as $\sqrt{J} = \m$) is $N:= n  \qdeg f - 2  \sum \qdeg x_i$, and so $\left[ R/J \right]_d = 0$  whenever $d \geq N+1$.  It follows that $\Deg{R}{\geq  N + 1} \subseteq J$, and to establish the claim, it suffices to show that
\begin{equation}
 \label{trick: e}
 \Colon{\bracket{\m}{e}}{\Deg{R}{\geq N+1}} \setminus \ \bracket{\m}{e} \subseteq \Deg{R}{\geq (p^e-1) \cdot \sum \qdeg x_i - N} = \Deg{R}{\geq (p^e + 1) \cdot \sum \qdeg x_i - n \cdot \qdeg f}.
\end{equation}

By means of contradiction, suppose \eqref{trick: e} is false.  Consequently, there exists a monomial \[\mu = x_1^{p^e-1 - s_1} \cdots x_n^{p^e-1-s_n} \in \Colon{\bracket{\m}{e}}{\Deg{R}{\geq N+1}}\] such that $\qdeg \mu \leq (p^e-1) \cdot \qdeg x_i - (N+1)$. This condition implies that the monomial $\mu_{\circ} := x_1^{s_1}\cdots x_n^{s_n}$ is in $\Deg{R}{\geq N+1}$, and as $\mu \in \Colon{\bracket{\m}{e}}{\Deg{R}{\geq N+1}}$, it follows that $\mu  \mu_{\circ}$ (which is apparently equal to $\(x_1\cdots x_n\)^{p^e-1})$ is in $\bracket{\m}{e}$, a contradiction.
\end{proof}

\begin{proposition}
\label{nubounds: P} 
In the setting of Setup \ref{isoSing: setup}, if $p \nmid \( \new{e}+1 \)$, then $\new{e} \geq \up{(p^e+1) \cdot \frac{ \sum \qdeg x_i }{ \qdeg f}-n}$. 
\end{proposition}
\begin{proof} The Leibniz rule shows that $\partial_i\(\bracket{\m}{e} \) \subseteq \bracket{\m}{e}$, and so differentiating $f^{\new{e}+1} \in \bracket{\m}{e}$ shows that 
$ ( \new{e} + 1 ) \cdot f^{\new{e}} \cdot  \df{i} \in \bracket{\m}{e} \text{ for all $i$}$.
Our assumption that $p \nmid \new{e} + 1$ then implies that
$ f^{\new{e}} \in \Colon{\bracket{\m}{e}}{J} \setminus \bracket{\m}{e} \subseteq \Deg{R}{\geq (p^e + 1) \cdot \sum \qdeg x_i - n \cdot \qdeg f}$,
where the exclusion follows by definition, and the final containment by Lemma \ref{Anurag'sTrick: L}.  Therefore,  
\[ \qdeg f \cdot \new{e} \geq (p^e + 1) \cdot \sum \qdeg x_i - n \cdot \qdeg f,\] and the claim follows.
\end{proof}

\begin{corollary}
\label{nubounds: C}
In the setting of Setup \ref{isoSing: setup}, write $\lambda = \min \set{ \frac{ \sum \qdeg x_i}{\qdeg f},1} = \frac{a}{b}$, in lowest terms.  If $\digit{\fpt{f}}{e}$, the $e^{\th}$ digit of $\fpt{f}$, is not equal to $p-1$, then 
 \[ p^e \tr{\lambda}{e} - p^e \tr{\fpt{f}}{e} \leq n - \up{ \frac{ \lpr{\paren{a}{p^e}}{b} + a}{b} }.\] 

\end{corollary}

\begin{proof}
By Proposition \ref{Truncation: P}, $\new{e} = p^e  \tr{\fpt{f}}{e} \equiv \digit{\fpt{f}}{e} \bmod p$, and so the condition that $\digit{\fpt{f}}{e} \neq p-1$ is equivalent to the condition that $p \nmid \(\new{e} + 1\)$.  In light of this, we are free to apply Proposition \ref{nubounds: P}. 
In what follows, we set $\delta:= \(\sum \qdeg x_i\) \cdot (\qdeg f)^{-1}$.

First, suppose that $\min \set{ \delta, 1} = 1$, so that $a = b = 1$.    Then  $ \up{ \( \lpr{\paren{a}{p^e}}{b} + a \) \cdot b^{-1} } = 2,$ and so it suffices to show that $p^e \tr{1}{e} - p^e \tr{\fpt{f}}{e} \leq n-2$.    However, the assumption that $\min \set{ \delta, 1} = 1$ implies that $\delta \geq 1$, and Proposition \ref{nubounds: P} then shows that \begin{align*}  
p^e \cdot \tr{\fpt{f}}{e} = \new{e} \geq  \up{ (p^e+1) \cdot \delta - n} \geq \up{p^e+1 - n}  &= p^e - 1 + 2 - n \\ 
 & = p^e \cdot \tr{1}{e} + 2-n.
\end{align*}

If instead $\min \set{ \delta, 1} = \delta$, Proposition \ref{nubounds: P} once again shows that
\begin{align*} p^e \tr{\fpt{f}}{e} = \new{e} \geq \up{ \(p^e + 1\) \cdot \delta  - n} & = \up{p^e \cdot \delta + \delta - n } \\ 
& = \up{ p^e \cdot \( \tr{\delta}{e} + \frac{ \lpr{\paren{a}{p^e}}{b}}{b \cdot p^e} \) + \delta - n} \\ 
& = p^e \cdot \tr{\delta}{e} + \up{ \frac{\lpr{\paren{a}{p^e}}{b}}{b} + \delta} - n,
\end{align*}
the second to last equality following from Lemma \ref{explicittruncation: L}.
\end{proof}

\begin{example}[Illustrating that Corollary \ref{nubounds: C} is not an equivalence]
\label{WeirdLValues: E}
If $p=2$ and $f$ is any $\LL^{\ast}$-linear combination of $x_1^{15}, \cdots, x_5^{15}$, Corollary \ref{nubounds: C} states that if $\digit{\fpt{f}}{e} \neq 1$, then  $\Delta_e:= 2^e \tr{3^{-1}}{e} - 2^e \tr{\fpt{f}}{e} \leq 4$.  
We claim that the converse fails when $e=4$.  Indeed, a direct computation, made possible by \cite{Diagonals},  shows that $\fpt{f} = \frac{1}{8}$, and comparing the base $2$ expansions of $\fpt{f} = \frac{1}{8}$ and $\lambda = \frac{1}{3}$ shows that $\Delta_4 = 4$, even though $\digit{\fpt{f}}{4} = 1 = p-1$.
\end{example}

\subsection{Proofs of the main results} \label{mainProofs:  SS}

In this subsection, we return to the statements in Section \ref{fptDiscussion: Section} whose proofs were postponed.  For the benefit of the reader, we restate these results here.

\begin{MainR}
Fix an $\NN$-grading on $R$.  Consider a \qh polynomial $f$ with $\sqrt{\jac{f}} = \m$, and write  $\lambda:=\min \set{ \frac{ \sum \qdeg x_i}{\qdeg f}, 1} = \frac{a}{b}$ in lowest terms.
\begin{enumerate}
\item \label{MT1: p} If $\fpt{f} \neq \lambda$, then \[  \fpt{f} = \lambda  - \( \frac{ \lpr{\paren{a}{p^L}}{b} + b \cdot E }{b \cdot p^L}\) = \tr{ \lambda }{L} - \frac{E}{p^L} \]  for some $(L,E) \in \NN^2$ with $L \geq 1$ and $0 \leq E \leq n-1-\up{\frac{\lpr{\paren{a}{p^L}}{b} + a}{b}}$.
\item \label{MT2: p} If $p > (n-2) \cdot b$ and $p \nmid b$, then $1 \leq L \leq \order{p}{b}$; note that $\order{p}{1} = 1$. 
\item \label{MT3: p}
If $p > (n-2) \cdot b$ and $p > b$,  then $ a < \lpr{\paren{a}{p^e}}{b} \text{ for all } 1 \leq e \leq L-1.$ 
\item \label{MT4: p} If $p> (n-1) \cdot b$, then there exists a unique pair $(L,E)$ satisfying the conclusions of \eqref{MT1: p}.
\end{enumerate}
\end{MainR}

\begin{proof} We begin by establishing \eqref{MT1: p}:  The two descriptions of $\fpt{f}$ are equivalent by Lemma \ref{explicittruncation: L}, and so it suffices to establish the identity in terms of truncations.  Setting $\Delta_e := p^e \tr{\lambda}{e} - p^e \tr{\fpt{f}}{e}$, Corollary \ref{boundsOnFirstDiff: C}, states that $\set{ \Delta_e}_{e \geq 1}$ is a non-negative, non-decreasing, unbounded sequence; in particular, $\min \set{ e : \Delta_e \neq 0 } \geq 1$ is well-defined, and we claim that \[ \ell:= \min \set{ e : \Delta_e \neq 0 } \leq L := \max \set{ e : \digit{\fpt{f}}{e} \neq p-1},\] the latter also being well-defined.  Indeed, set $\mu_e:= \up{ \frac{ \lpr{\paren{a}{p^e}}{b} + a}{b} }$.  As $1 \leq \mu_e \leq 2$, the sequence $\set{ n - \mu_e}_{e \geq 1}$ is bounded above by $n-1$, and therefore $\Delta_e > n - \mu_e$ for $e \gg 0$.  For such $e \gg 0$, Corollary \ref{nubounds: C} implies that $\digit{\fpt{f}}{e} = p-1$, which demonstrates that $L$ is well-defined.  Note that, by definition, $\Delta_{\ell} = \digit{\lambda}{\ell} - \digit{\fpt{f}}{\ell} \geq 1$,  so that $\digit{\fpt{f}}{\ell} \leq \digit{\lambda}{\ell} - 1 \leq p - 2$;  by definition of $L$, it follows that $\ell \leq L$.

As $\digit{\fpt{f}}{e} = p-1$ for $e \geq L+1$,  
\begin{equation}
\label{restatementMR: e}
\fpt{f}  = \tr{\fpt{f}}{L} + \frac{1}{p^L} = \tr{\lambda}{L} - \frac{\Delta_L}{p^L} + \frac{1}{p^L} = \tr{\lambda}{L} - \frac{E}{p^L},   
\end{equation} 
where $E := \Delta_L - 1$.  In order to conclude this step of the proof, it suffices to note that  
\begin{equation}
\label{3ineq: e}
1 \leq \Delta_{\ell} \leq \Delta_L \leq n - \mu_L \leq n-1;
\end{equation}
indeed, the second bound in \eqref{3ineq: e} follows from the fact that $L \geq \ell$, the third follows from Corollary \ref{nubounds: C}, and the last from the bound $1 \leq \mu_e \leq 2$. 

For point \eqref{MT2: p}, we continue to use the notation adopted above.  We  begin by showing that 
\begin{equation}
\label{deltavanishing: e} 
\Delta_{e} = 0 \text{ for all $0 \leq e \leq L-1$ whenever } p > (n-2) \cdot b.
\end{equation}
As the sequence $\Delta_e$ is non-negative and non-decreasing, it suffices to show that $\Delta_{L-1} = 0$.  Therefore, by way of contradiction, we suppose that $\Delta_{L-1} \geq 1$.  By definition $0 \leq \digit{\fpt{f}}{L} \leq p-2$, and hence 
\[ \Delta_L = p \cdot \Delta_{L-1} + \digit{\lambda}{L} - \digit{\fpt{f}}{L} \geq \digit{\lambda}{L} + 2.\]
Comparing this with \eqref{3ineq: e} shows that $\digit{\lambda}{L} + 2 \leq \Delta_L \leq n-1$, so that 
\[ \digit{\lambda}{L} \leq n-3.\]  
On the other hand, if $p > (n-2) \cdot b$, then it follows from the explicit formulas in Lemma \ref{explicittruncation: L} that 
\begin{equation} 
\label{largedigit: e}
\digit{\lambda}{e} = \frac{\lpr{\paren{a}{p^{e-1}}}{b} \cdot p - \lpr{\paren{a}{p^e}}{b}}{b} \geq \frac{p-b}{b} > \frac{(n-2)\cdot b - b}{b} = n-3 \text{ for every $e \geq 1$}. 
\end{equation}
In particular, setting $e = L$ in this identity shows that $\digit{\lambda}{L} > n-3$, contradicting our earlier bound.  

Thus, we conclude that \eqref{deltavanishing: e} holds, which when combined with \eqref{3ineq: e} shows that $L = \min \set{ e : \Delta_e \neq 0 }$.  In summary, we have just shown that $L  = \ell$ when $p > (n-2) \cdot b$.  If we assume further that $p \nmid b$, the desired bound $L = \ell \leq \order{p}{b}$ then follows from Corollary \ref{boundsOnFirstDiff: C}.

We now focus on point \eqref{MT3: p}, and begin by observing that  
\begin{equation} 
\label{fptnte: e}
\fpt{f} = \digit{\lambda}{1} : \cdots : \digit{\lambda}{L-1} : \digit{\lambda}{L} - \Delta_L : \overline{p-1}  \ (\base p) \text{ whenever $p > (n-2) \cdot b$.} \end{equation}  
Indeed, by \eqref{deltavanishing: e}, the first $L-1$ digits of $\fpt{f}$ and $\lambda$ agree, while $\digit{\fpt{f}}{e} = p-1$ for $e \geq L+1$, by definition of $L$.  Finally, \eqref{deltavanishing: e} shows that $\Delta_L = \digit{\lambda}{L}  - \digit{\fpt{f}}{L}$, so that $\digit{\fpt{f}}{L} = \digit{\lambda}{L} - \Delta_L$.

Recall that, by the second point of Proposition \ref{Truncation: P}, the first digit of $\fpt{f}$ is its smallest digit, and it follows from \eqref{fptnte: e} that $\digit{\lambda}{1} \leq \digit{\lambda}{e} \text{ for all $1 \leq e \leq L$,  with this inequality being strict for $e=L$.}$  However, it follows from the explicit formulas in Lemma \ref{explicittruncation: L} that whenever $p > b$, 
\begin{align*}
\digit{\lambda}{1} \leq \digit{\lambda}{e} & \iff a \cdot p - \lpr{\paren{a}{p}}{b} \leq \lpr{\paren{a}{p^{e-1}}}{b} \cdot p - \lpr{\paren{a}{p^e}}{b} \\ 
       & \iff a \leq \lpr{\paren{a}{p^{e-1}}}{b},
\end{align*}
where the second equivalence relies on the fact that $p>b$.  Summarizing, we have just shown that $a \leq \lpr{\paren{a}{p^{e-1}}}{b}$ for all $1 \leq e \leq L$  whenever $p > (n-2) \cdot b$ and $p> b$; relabeling our index, we see that \[ a \leq \lpr{\paren{a}{p^e}}{b} \text{ for all } 0 \leq e \leq L-1 \text{ whenever } p > (n-2) \cdot b \text{ and } p> b.\] 
It remains to show that this bound is strict for $1 \leq e \leq L-1$.  
By contradiction, assume that $a = \lpr{\paren{a}{p^e}}{b}$ for some such $e$.  In this case, $a \equiv a \cdot p^e \bmod b$, and as $a$ and $b$ are relatively prime, we conclude that $p^e \equiv 1 \bmod b$, so that $\order{p}{b} \mid e$.  However, by definition $ 1 \leq e \leq L-1 \leq \order{p}{b} -1$, where the last inequality follows point \eqref{MT2: p}.  Thus, we have arrived at a contradiction, and therefore conclude that our asserted upper bound is strict for $1 \leq e \leq L-1$.

To conclude our proof, it remains to establish the uniqueness statement in point \eqref{MT4: p}.  To this end, let $({L'}, {E'})$ denote any pair of integers satisfying the conclusions of point \eqref{MT1: p} of this Theorem;  that is, \[ \fpt{f} = \tr{\lambda}{{L'}} - {E'} \cdot p^{-{L'}} \text{ with } 1 \leq {E'} \leq n- 1 - \mu_{{L'}} \leq n-2.\]  
A modification of \eqref{largedigit: e} shows that $\digit{\lambda}{e} > n-2$, and hence that $\digit{\lambda}{e} \geq E' + 1$, whenever $p > (n-1) \cdot b$, and it follows that 
\[ \fpt{f} = \tr{\lambda}{{L'}} - {E'} \cdot p^{-{L'}} = . \digit{\lambda}{1} : \cdots : \digit{\lambda}{L'-1} : \digit{\lambda}{L'} - (E + 1) : \overline{p-1} \text{ whenever $p>(n-1) \cdot b$.}\]  
The uniqueness statement then follows from comparing this expansion with \eqref{fptnte: e} and invoking the uniqueness of non-terminating base $p$ expansions.
\end{proof}

\begin{UniformLBoundsR}
In the setting of Theorem \ref{Main: T}, if $\fpt{f} \neq \lambda$ and $p \nmid b$, then $p^{M} \cdot \fpt{f} \in \NN$, where $M :=2 \cdot \phi(b) + \up{\log_2(n-1)},$ and $\phi$ denotes Euler's phi function.
\end{UniformLBoundsR}

\begin{proof}  We adopt the notation used in the proof of Theorem \ref{Main: T}.  In particular, $\ell \leq L$ and $\fpt{f} \in p^{-L} \cdot \NN$.  Setting $s = \order{p}{b}$, and $k = \up{\log_p(n-1)}$ in Lemma \ref{Difference: L} shows that
\begin{equation}
\label{wrongwayineq: e} \Delta_{\ell + s + \up{\log_p(n-1)}} \geq p^{\up{\log_p(n-1)}} + 1 \geq n.
\end{equation}
 
By definition of $L$, Corollary \ref{nubounds: C} states that $\Delta_L \leq n-1$, and  as $\set{ \Delta_e}_{e \geq 1}$ is non-decreasing, \eqref{wrongwayineq: e} then shows that $L$ is bounded above by $\ell + s + \up{ \log_p(n-1)}$.  To obtain a uniform bound, note that $\ell \leq s$, by Corollary \ref{boundsOnFirstDiff: C}, while $s \leq \phi(b)$, by definition, and $\log_p(n-1) \leq \log_2(n-1)$, as $p \geq 2$.
\end{proof}

\begin{SecondaryR}
In the context of Theorem \ref{Main: T}, suppose that $\sum \qdeg x_i > \qdeg f$, so that $\rho:= \frac{\sum \qdeg x_i}{\qdeg f}$ is greater than $1$.  If $p> \frac{n-3}{\rho - 1}$, then $\fpt{f} = 1$.
\end{SecondaryR}

\begin{proof}  We begin with the following elementary manipulations, the first of which relies on the assumption that $\rho - 1$ is positive:  Isolating $n-3$ in our assumption that $p > (n-3) \cdot \(\rho - 1 \)^{-1} - 1$ shows that $(p+1) \cdot (\rho - 1) > n -3$, and adding $p+1$ and subtracting $n$ from both sides then shows that $(p+1) \cdot \rho - n > p - 2$; rounding up, we see that \begin{equation} \label{simplederivation: e} \up{(p+1) \cdot \rho - n} \geq p-1.\end{equation}  

Assume, by means of contradiction, that $\fpt{f} \neq 1$.  By hypothesis, $1 = \min \set{ \rho , 1 }$, and Corollary \ref{boundsOnFirstDiff: C} then states that $1 = \min \set{ e : p^e \tr{1}{e} - p^e \tr{\fpt{f}}{e} \geq 1}$; in particular, 
\[ \new{} = \digit{\fpt{f}}{1} = p \cdot \tr{\fpt{f}}{1} \leq p \tr{1}{1} -1 = p-2.\] However, this bound allows us to apply Proposition \ref{nubounds: P}, which when combined with \eqref{boundsOnFirstDiff: C}, implies that 
\[ \new{} \geq \up{(p+1) \cdot \rho - n } \geq p-1.\] Thus, we have arrived at a contradiction, which allows us to conclude that $\fpt{f}=1$.
\end{proof}

\section{Applications to log canonical thresholds}
\label{LCT: S}

Given a polynomial $\fQ$ over $\Q$, we will denote its \emph{log canonical threshold} by $\lct{\fQ}$. In this article, we will not need to refer to the typical definition(s) of $\lct{\fQ}$ (e.g., via resolution of singularities), and will instead rely on the limit in \eqref{fptlct: e} below as our definition.  However, so that the reader unfamiliar with this topic may better appreciate \eqref{fptlct: e}, we present the following characterizations.  In what follows, we fix $\fQ \in \Q[x_1, \cdots, x_n]$.   

\begin{enumerate}

\item  If $\pi: X \to  \mathbb{A}^n_{\Q}$ is a \emph{log resolution} of the pair $\( \mathbb{A}^n_{\Q}, \mathbb{V}(\fQ)\)$, then $\lct{\fQ}$ is the supremum over all $\lambda > 0$ such that the coefficients of the divisor $K_{\pi} - \lambda \cdot \pi^{\ast} \operatorname{div} (f)$ are all greater than $-1$; here, $K_{\pi}$ denotes the relative canonical divisor of $\pi$.  

\item For every $\lambda > 0$,  consider the function $\Gamma_{\lambda}(\fQ) : \mathbb{C}^n\to \RR$ given by \[(z_1, \cdots, z_n) \mapsto \left | f(z_1, \cdots, z_n) \right|^{-2 \lambda},\] where $| \cdot | \in \RR$ denotes the norm of a complex number; note that $\Gamma_{\lambda}(\fQ)$ has a pole at all (complex) zeros of $\fQ$.  In this setting, $\lct{\fQ} : = \sup \set{ \lambda : \Gamma_{\lambda}(\fQ) \text{ is locally $\RR$-integrable}},$ where here, ``locally $\RR$-integrable'' means that we identify $\mathbb{C}^n = \RR^{2n}$, and require that this function be (Lebesque) integrable in a neighborhood of every point in its domain.

\item {{The roots of the Bernstein-Sato polynomial $b_{\fQ}$ of $\fQ$ are all negative rational numbers, and $-\lct{\fQ}$ is the largest such root \cite{KollarPairs}.}}

\end{enumerate}

For more information on these invariants, the reader is referred to the surveys  \cite{BL2004, EM2006}.  We now recall the striking relationship between $F$-pure and log canonical thresholds:  Though there are many results due to many authors relating characteristic zero and characteristic $p>0$ invariants, the one most relevant to our discussion is the following theorem, which is due to Musta\c{t}\u{a}  and the fourth author.

\begin{theorem}
\label{MZuniformDiff: T}
 \cite[Corollary 3.5, 4.5]{MZ2012} \label{MZ}
Given an polynomial $\fQ$ over $\Q$, there exist constants $C \in \mathbb{R}_{>0}$ and $N \in \NN$ (depending only on $\fQ$) with the following property:  For $p \gg 0$, either $\fpt{\fp} = \lct{\fQ}$, or 
\[ \frac{1}{p^N} \leq \lct{\fQ} - \fpt{\fp} \leq \frac{C}{p}. \] 
\end{theorem}

Note that, as an immediate corollary of Theorem \ref{MZuniformDiff: T}, 
\begin{equation} 
\label{fptlct: e}
\fpt{\fp} \leq \lct{\fQ} \text{ for all } p \gg 0 \text{ and } \lim_{p \to \infty} \fpt{\fp} = \lct{\fQ}.  
\end{equation}
We point out that \eqref{fptlct: e} (which follows from the work of Hara and Yoshida) appeared in the literature well before Theorem \ref{MZuniformDiff: T} (see, e.g., \cite[Theorem 3.3, 3.4]{MTW2005}).

\subsection{Regarding uniform bounds}
Though the constants $C \in \mathbb{R}_{>0}$ and $N \in \NN$ appearing in Theorem \ref{MZuniformDiff: T} are known to depend only on $\fQ$, their determination is complicated (e.g., they depend on numerical invariants coming from resolution of singularities), and are therefore not explicitly described.  In Theorem \ref{DifferenceBounds: T} below, we give an alternate proof of this result for \qh polynomials with an isolated singularity at the origin;  in the process of doing so, we  also identify explicit values for $C$ and $N$.

\begin{theorem}
\label{DifferenceBounds: T}
If  $f_\Q \in \Q[x_1, \cdots, x_n]$ is \qh under some $\NN$-grading, with $\sqrt{\jac{\fQ}} = \m$, then $\lct{\fQ} = \min \set{ \frac{ \sum \qdeg x_i}{\qdeg f}, 1}$, which we write as $\frac{a}{b}$ in lowest terms.  Moreover, if $\fpt{f_p} \neq \lct{\fQ}$, then \[\frac{b^{-1}}{p^{\order{p}{b}}}\leq \lct{\fQ}- \fpt{\fp} \leq \frac{n-1-b^{-1}}{p} \text{ for $p \gg 0$},\]
where  $\order{p}{b}$ denotes the order of $p$ mod $b$ (which equals one when $b=1$, by convention).
\end{theorem}

\begin{proof}
As the reduction of $\df{k}$ mod $p$ equals $\partial_k(\fp)$ for large values of $p$, the equality $\sqrt{\jac{\fQ}} = \m$ reduces mod $p$ for $p \gg 0$.  Taking $p \to \infty$, it follows from Theorem \ref{Main: T} and \eqref{fptlct: e} that $\lct{\fQ} = \min \set{ \frac{ \sum \qdeg x_i}{\qdeg f}, 1}$, and in light of this, Theorem \ref{Main: T} states that  
\begin{equation}
\label{uniformb: e}
\lct{\fQ} - \fpt{f_p} = \frac{ \lpr{\paren{a}{p^L}}{b}}{b \cdot p^L} + \frac{E}{p^L}.
\end{equation}
 
If $\lct{\fQ} \neq 1$, then  \[ \frac{1}{b \cdot p^L} \leq \frac{\lpr{\paren{a}{p^L}}{b}}{b \cdot p^L} \leq \frac{1-b^{-1}}{p^L}.\]  Furthermore, Theorem \ref{Main: T} implies that $1 \leq L \leq \phi(b) \text{ and } 0 \leq E \leq n-2$  for $p \gg 0$.
If instead $\lct{\fQ} = a = b = 1$, then $\lpr{\paren{a}{p^L}}{b} = \phi(b) = 1$, and hence\[ \frac{\lpr{\paren{a}{p^L}}{b}}{b \cdot p^L} = \frac{b^{-1}}{p^L}.\]  Moreover, in this case, Theorem \ref{Main: T} shows that $L = 1 \text{ and } 0 \leq E \leq n-3$ for $p \gg 0$. Finally, it is left to the reader to verify that substituting these inequalities into \eqref{uniformb: e} produces the desired bounds in each case.
\end{proof}

\begin{remark}[On uniform bounds]
Of course, $\order{p}{b} \leq \phi(b)$, where $\phi$ denotes Euler's phi function.  By enlarging $p$, if necessary, it follows that the lower bound in Theorem \ref{DifferenceBounds: T} is itself bounded below by $p^{-\phi(b)-1}$.  In other words, in the language of Theorem \ref{MZuniformDiff: T}, we may take $N= \phi(b) + 1$ and $C= n-1-b^{-1}$.
\end{remark}

\begin{remark}[Regarding sharpness]
\label{SharpBounds: R} The bounds appearing in Theorem \ref{DifferenceBounds: T} are sharp:  If $d > 2$ and $\fQ = x_1^d + \cdots + x_d^d$, then $\lct{\fQ} = 1$, and Theorem \ref{DifferenceBounds: T} states that 
\begin{equation} \label{diffexam: e} \frac{1}{p} \leq \lct{\fQ} - \fpt{\fp} \leq \frac{d-2}{p} \end{equation}
 whenever $\fpt{\fp} \neq 1$ and $p \gg 0$.  However, it is shown in \cite[Corollary 3.5]{Diagonals} that \[ \lct{\fQ} - \fpt{\fp} = 1 - \fpt{\fp} =  \frac{\lpr{p}{d} - 1}{p}\] whenever $p > d$.  If $d$ is odd and $p \equiv 2 \bmod d$, then the lower bound in \eqref{diffexam: e} is obtained, and  similarly, if $p \equiv d-1 \bmod d$, then the upper bound in \eqref{diffexam: e} is obtained; in both these cases, Dirichlet's theorem guarantees that there are infinitely many primes satisfying these congruence relations.
\end{remark}

\subsection{On the size of a set of bad primes}
\label{BadPrimes: SS}

In this subsection, we record some simple observations regarding the set of primes for which the $F$-pure threshold does \emph{not} coincide with the log canonical threshold, and we begin by recalling the case of elliptic curves:  Let $\fQ \in \Q[x,y,z]$ be a homogeneous polynomial of degree three with $\sqrt{ \jac{\fQ}} = \m$, so that $E := \mathbb{V}(f)$ defines an elliptic curve in $\mathbb{P}^2_{\Q}$.  As shown in the proof of Theorem \ref{DifferenceBounds: T}, the reductions $\fp \in \FF_p[x, y, z]$ satisfy these same conditions for $p \gg 0$, and thus define elliptic curves $E_p = \mathbb{V} (f_p) \subseteq \mathbb{P}^2_{\FF_p}$ for all $p \gg 0$.  Recall that the elliptic curve $E_p$ is called \emph{supersingular} if the natural Frobenius action on the local cohomology module $H^2_{(x,y,z)} \( \FF_p[x,y,z] / (f_p) \)$ is injective, or equivalently, if $(f_p)^{p-1} \notin (x^p, y^p, z^p)$ (see, e.g, \cite[Chapters V.3 and V.4]{Silverman} for these and other characterizations of supersingularity). Using these descriptions, one can show that $E_p$ is supersingular if and only if $\fpt{f_p} = 1$ \cite[Example 4.6]{MTW2005}.  In light of this, Elkies' well-known theorem on the set of supersingular primes, which states that $E_p$ is supersingular for infinitely many primes $p$, can be restated as follows.

\begin{theorem} \cite{Elkies} \ 
If $\fQ \in \Q[x,y,z]$ is as above, the set of primes $\set{ p : \fpt{\fp} \neq \lct{\fQ}}$ is infinite.
\end{theorem}

Recall that given a set $S$ of prime numbers, the density
of $S$, $\delta(S)$, is defined as \[
\delta(S) = \lim_{n \to \infty}\frac{^{\#} \set{ p \in S : p \leq
    n}}{^{\#} \set{p : p \leq n}}. \]

In the context of elliptic curves over $\Q$, the set of primes $\set{ p :
  \fpt{f_p} \neq \lct{\fQ}}$, which is infinite by Elkies' result,  may be quite large (i.e., have density $\frac{1}{2}$), or may be quite small (i.e., have density zero);  see \cite[Example 4.6]{MTW2005} for more information.   This discussion motivates the following question.

\begin{question}  
\label{badprimes: Q}
For which polynomials $\fQ$ is the set of primes $\set{ p : \fpt{f_p} \neq \lct{\fQ} }$ infinite?  In the case that this set is infinite, what is its density?
\end{question}

As illustrated by the case of an elliptic curve, Question \ref{badprimes: Q} is quite subtle, and one expects it to be quite difficult to address in general.  However, as we see below, when the numerator of $\lct{\fQ}$ is not equal to $1$, one is able to give a partial answer to this question using simple methods.  Our main tool will be Proposition \ref{Truncation: P}, which provides us with a simple criterion for disqualifying a rational number from being an $F$-pure threshold.  We stress the fact that Proposition \ref{BadPrimes: P} is not applicable when $\lct{\fQ} = 1$, and hence sheds no light on the elliptic curve case discussed above.

\begin{proposition}
\label{BadPrimes: P}
Let $\fQ$ denote any polynomial over $\Q$, and write $\lct{\fQ} =
\frac{a}{b}$ in lowest terms.  If $a \neq 1$, then the set of primes for which $\lct{\fQ}$ is not an $F$-pure threshold (of \emph{any} polynomial) is infinite, and contains all
primes $p$ such that $p^e \cdot a \equiv 1 \bmod b$ for some
$e \geq 1$.  In particular,\[ \delta \( \set{ p : \fpt{f_p} \neq
  \lct{\fQ}} \) \geq \frac{1}{\phi(b)}.\]
\end{proposition}

\begin{proof}
As $a$ and $b$ are relatively prime, there exists $c \in \NN$ such that $a \cdot c \equiv 1 \bmod b$.  We claim that 
\begin{align*} \set{ p : p \equiv c \bmod b} & \subseteq  \set { p : p^e \cdot a \equiv 1 \bmod  b \text{ for some $e \geq 1$}} \\ & \subseteq \set{ p : \lct{\fQ} \text{ is not an $F$-pure threshold in characteristic $p>0$}}. \end{align*}
Once we establish this, the proposition will follow, as $\delta \(\set{ p : p \equiv c \bmod b}\) = \frac{1}{\phi(b)}$ by Dirichlet's theorem.  By definition of $c$, the first containment holds by setting $e=1$, and so it suffices to establish the second containment.  However, if $p^e \cdot a \equiv 1 \bmod b$ for some $e \geq 1$, then Lemma \ref{explicittruncation: L} shows that \[\digit{\lct{\fQ}}{e+1} = \frac{\lpr{\paren{a}{p^e}}{b} \cdot p - \lpr{\paren{a}{p^{e+1}}}{b}}{b} = \frac{ p - \lpr{\paren{a}{p^{e+1}}}{b}}{b}.\]  On the other hand, Lemma \ref{explicittruncation: L} also shows that \[   \digit{\lct{\fQ}}{1} = \frac{ a \cdot p - \lpr{\paren{a}{p}}{b} }{b},\] and as $a \geq 2$, by assumption, we see that $\digit{\lct{\fQ}}{1} > \digit{\lct{\fQ}}{e}$ for all $p \gg 0$.  In light of this, the second point of Proposition \ref{Truncation: P}, which shows that the first digit of an $F$-pure threshold must be its smallest, shows that $\lct{\fQ}$ could not be the $F$-pure threshold of \emph{any} polynomial in characteristic $p>0$.
\end{proof}

We conclude this section with the following example, which follows immediately from Corollary \ref{aCY: C}, and which illustrates a rather large family of polynomials whose set of ``bad" primes greatly exceeds the bound given by Proposition \ref{BadPrimes: P}.

\begin{example} \label{LargeClassBadPrimes: E}
If $\fQ \in \Q[x_1, \cdots, x_{d-1}]$ is homogeneous (under the standard grading) of degree $d$ with $\sqrt{\jac{\fQ}} = \m$, then $\set{ p : p \not \equiv 1 \bmod d} \subseteq \set{ p : \fpt{f_p} \neq \lct{\fQ} = 1- \frac{1}{d}}$.  In particular, \[ \delta \( \set{ p : \fpt{f_p} \neq \lct{\fQ}} \) \geq \delta \(\set{ p : p \not \equiv 1 \bmod d} \)= 1 - \frac{1}{\phi(d)}.\]
\end{example}

\section{A special case of ACC and local $\mathfrak{m}$-adic constancy for $F$-pure thresholds}

Motivated by the relationship between $F$-pure thresholds and log canonical thresholds,  Blickle, Musta\c{t}\u{a}, and Smith conjectured the following.

\begin{conjecture}
\label{ACC: Con}
\cite[Conjecture 4.4]{BMS2009} \ Fix an integer $n \geq 1$.
\begin{enumerate}
\item The set $\set{ \fpt{f} : f \in \LL[x_1, \cdots, x_n]}$ satisfies the ascending chain condition (ACC); i.e., it contains no strictly increasing, infinite sequence.   
\item  For every $f \in \LL[x_1, \cdots, x_n]$, there exists an integer $N$ (which may depend on $f$) such that \[ \fpt{f} \geq \fpt{f+g} \text{ for all $g \in \m^N$}. \]
\end{enumerate}
As discussed in \cite[Remark 4.5]{BMS2009}, the first conjecture implies the second, which states that the $F$-pure threshold function $f \mapsto \fpt{f}$ is locally constant (in the $\m$-adic topology).
\end{conjecture}

In this section, we confirm the first conjecture for a restricted set of $F$-pure thresholds (see Proposition \ref{ACC: P}).  Additionally, we confirm the second in the case that $f$ is \qh under some $\NN$-grading with $\sqrt{ \jac{f}} = \m$ (see Propositions \ref{higherdegreefpt: P} and \ref{madicusc: P}).

\subsection{A special case of ACC}

\begin{definition}
 For every $\wt{} \in \N^n$, let $W_{\wt{}}$ denote the set of polynomials $f \in \LL[x_1, \cdots , x_n]$ satisfying the following conditions:
 
\begin{enumerate}
\item $\sqrt{\jac{f}} = \m$.
\item $f$ is \qh under the grading determined by $\( \qdeg x_1, \cdots, \qdeg x_n \) =\wt{}$.
\item $p \nmid \qdeg f$ (and hence, does not divide the denominator of $\min \set{ \frac{ \sum \qdeg x_i}{\qdeg f}, 1}$, in lowest terms).
\end{enumerate}
Given $N \in \N$, set $W_{\preccurlyeq N} := \bigcup_{\wt{} } W_{\wt{}}$, where the union is taken over all $\wt{} = (\wt{1}, \ldots, \wt{n})\in \N^n$ with $\wt{i} \leq N$ for each $1 \leq i \leq n$.
\end{definition}

\begin{proposition}   
\label{ACC: P}
For every $N \in \N$ and $\mu \in (0,1]$, the set 
\[ \set{ \fpt{f} : f \in W_{\preccurlyeq N} } \cap (\mu, 1] \] is finite.  In particular, this set of $F$-pure thresholds satisfies ACC.
\end{proposition}

\begin{proof}
Fix $f \in W_{\preccurlyeq N}$ such that $\fpt{f} > \mu$.  By definition,  there exists an $\N$-grading on $\LL[x_1, \cdots, x_n]$ such that $\qdeg x_i \leq N$ for all $1 \leq i \leq N$, and under which $f$ is \qhns.  Moreover, by Lemma \ref{generalfptBound: L},
 \[ \mu < \fpt{f} \leq \frac{\sum_{i=1}^n \qdeg x_i}{\qdeg f} \leq \frac{n \cdot N}{\qdeg f}. \]  Consequently,  $\qdeg f \leq \frac{n \cdot N}{\mu}$, and it follows that \[ \lambda:= \min \set{\frac{\sum_{i=1}^n \qdeg x_i}{\qdeg f}, 1} \subseteq {S}:=(0,1] \cap \set{ \frac{a}{b} \in \Q : b \leq  \frac{n \cdot N}{\mu}},\] a finite set.  We will now show that $\fpt{f}$ can take on only finitely many values:  If $\fpt{f} \neq \lambda$, then by Corollary \ref{UniformLBounds: C}, there exists an integer $M_{\lambda}$, depending only on $\lambda$ and $n$, such that $p^{M_{\lambda}} \cdot \fpt{f} \in \NN$. If $\sM : = \max \set{ M_{\lambda} : \lambda \in {S}}$, it follows that $\fpt{f} \in \set{ \frac{a}{p^{\sM}} : a \in \N} \cap (0,1]$, a finite set.
\end{proof}

\subsection{A special case of local $\m$-adic constancy of the $F$-pure threshold function.}

Throughout this subsection, we fix an $\NN$-grading on $\LL[x_1, \cdots, x_n]$.

\begin{lemma}
\label{fptfg<fptf: L} Consider $f \in \m$ such that $p^L \cdot \fpt{f} \in \NN$ for some $L \in \NN$.  If $g \in \m$, then 
\[ \fpt{f+g} \leq \fpt{f} \iff (f+g)^{p^L \cdot \fpt{f}} \in \bracket{\m}{L}.\]
\end{lemma}

\begin{proof}
If $(f+g)^{p^L \cdot \fpt{f}} \in \bracket{\m}{L}$, then $(f+g)^{p^s \cdot \fpt{f}} \in \bracket{\m}{s}$ for $s \geq L$.  Consequently, $\newfg{s} < p^s \cdot \fpt{f}$ for $s \gg 0$, and hence $\fpt{f+g} \leq \fpt{f}$. We now focus on the remaining implication.  

By the hypothesis, $p^L \cdot \fpt{f} - 1 \in \NN$, and hence the identity $\fpt{f} = \frac{ p^L \cdot \fpt{f} - 1}{p^L} + \frac{1}{p^L}$ shows that  \[ \tr{\fpt{f}}{L} = \frac{p^L \cdot \fpt{f} - 1}{p^L}.\]  
If $\fpt{f+g} \leq \fpt{f}$, the preceding identity and Proposition \ref{Truncation: P} show that 
\[ \newfg{L} = p^L \tr{ \fpt{f+g} }{L} \leq p^L \tr{\fpt{f}}{L} = p^L \fpt{f} - 1,\]  
and consequently, this bound for $\newfg{L}$ shows that $(f+g)^{p^L \fpt{f} }  \in \bracket{\m}{L}$.
\end{proof}

\begin{lemma} \label{boundNotInBracket: L}
If $h$ is homogeneous and $h \notin \bracket{\m}{e}$, then $\qdeg h \leq (p^e -1 ) \cdot \sum_{i=1}^n \qdeg x_i$.
\end{lemma}

\begin{proof}
Every supporting monomial of $h$ is of the form $x_1^{p^e-a_1} \cdots x_n^{p^e-a_n}$, where each $a_i\geq 1$.  Then
\[\qdeg h = \sum_{i=1}^n (p^e-a_i)  \qdeg x_i  \leq (p^e-1)  \sum_{i=1}^n \qdeg x_i .\]

\end{proof}

\begin{proposition}
\label{fptf<fptfg: P}  
Fix $f\in\m$ \qhns.  If $g \in \Deg{R}{\geq \qdeg f+1}$, then $\fpt{f} \leq \fpt{f+g}$.   
\end{proposition}

\begin{proof}  
If suffices to show that  for every $e \geq 1$, $\new{e} \leq \newfg{e}$; i.e., if $N := \new{e}$, then $(f+g)^N \notin \bracket{\m}{e}$.   
Suppose, by way of contradiction, that $(f+g)^N = f^N + \sum_{k = 1}^N \binom{N}{k} f^{N-k}g^k \in \bracket{\m}{e}$;  note that, as $f^N \notin \bracket{\m}{e}$ by definition, each monomial summand of $f^N$ must cancel with one of $\sum_{k = 1}^N \binom{N}{k} f^{N-k}g^k$.  
However, for any monomial summand $\mu$ of any $f^{N-k} g^k$, $k \geq 1$, \[ \qdeg \mu \geq (N-k) \qdeg f + k( \qdeg f +1) > N \qdeg f = \qdeg(f^N),\] and such cancelation is impossible.
\end{proof}

\begin{lemma}
\label{SubtleExponent: L}
Fix $f\in \m$ \qh such that $\lambda := \frac{\sum \qdeg x_i }{\qdeg f} \leq 1$.
If $(p^e-1) \cdot  \lambda \in \N$ and $g \in \Deg{R}{\geq \qdeg f+1}$,  then $(f+g)^{p^e \tr{\lambda}{e}} \equiv f^{p^e \tr{\lambda}{e}} \bmod \bracket{\m}{e}$.
\end{lemma}

\begin{proof}  We claim that 
\begin{equation}
\label{subtle1: e}
f^{p^e \tr{\lambda}{e}-k} g^k \in \bracket{\m}{e} \text{ for all } 1 \leq k \leq p^e  \tr{\lambda}{e} . 
\end{equation}

Indeed, suppose that \eqref{subtle1: e} is false. 
As $g \in \Deg{R}{\geq \qdeg f+1}$ and $\mu$ is a supporting monomial 
of $f^{p^e \tr{\lambda}{e} - k}  g^k$, we also have that 
\begin{equation} \label{lowerDegBound:E}
\qdeg \mu \geq \qdeg f \cdot  \ \( p^e \tr{\lambda}{e}-k \) + (\qdeg f +1) \cdot k = \qdeg f \cdot  p^e \tr{\lambda}{e} + k.
\end{equation}
However, as $(p^e - 1) \cdot \lambda \in \N$, it follows from  Lemma \ref{BasicProperties: L} that $ p^e \tr{\lambda}{e} = (p^e - 1) \cdot  \lambda$. Substituting this into \eqref{lowerDegBound:E} shows that 
\[ \qdeg \mu \geq \qdeg f \cdot (p^e -1) \cdot  \lambda + k = k +  (p^e-1)  \sum_{i=1}^n \qdeg x_i, \] 
which contradicts Lemma \ref{boundNotInBracket: L} as $k \geq 1$. 
Thus, \eqref{subtle1: e} holds, and it follows from the Binomial Theorem that $(f+g)^{p^e \tr{\lambda}{e}} \equiv f^{p^e \tr{\lambda}{e}} \bmod \bracket{\m}{e}$.
\end{proof}

We are now able to prove our first result on the $\m$-adic constancy of the $F$-pure threshold function, which does not require the isolated singularity hypothesis.

\begin{proposition}
\label{higherdegreefpt: P}
Fix $f\in \m$ \qh such that $\lambda := \frac{\sum \qdeg x_i }{\qdeg f} \leq 1$, and suppose that either $\fpt{f} = \lambda$, or $\fpt{f} = \tr{\lambda}{\Ell}$ and $(p^{\Ell} - 1) \cdot \lambda \in \N$ for some $L\geq 1$.
Then $\fpt{f+g} = \fpt{f}$ for each $g \in \Deg{R}{\geq \qdeg f +1}$.
\end{proposition}

\begin{proof}
By Proposition \ref{fptf<fptfg: P}, it suffices to show that $\fpt{f} \geq \fpt{f+g}$.  First say that $\fpt{f}=\lambda$. 
It is enough to show that for all $e \geq 1$, $(f+g)^{\new{e}+1}\in \bracket{\m}{e}$, so that $\new{e} \geq \nu_{f+g}(p^e)$.
By the Binomial Theorem, it suffices to show that for all $0 \leq k \leq \new{f}+1$, $f^{\new{e}+1-k} g^{k} \in \bracket{\m}{e}$.
To this end, take any monomial $\mu$ of such an $f^{\new{e}+1-k} g^{k}$.  Then
\begin{equation} \label{lambdaGEQ: e}
\qdeg \mu \geq (\new{e}+1-k)\cdot \qdeg f + k \cdot (\qdeg f + 1) = (\new{e} + 1) \cdot \qdeg f + k \geq (\new{e} + 1) \cdot \qdeg f .
\end{equation}
By Lemma \ref{Truncation: P}, $\new{e} = p^e \tr{\lambda}{e}$, and by definition, $\tr{\alpha}{e} \geq \alpha - \frac{1}{p^e}$ for all $0 < \alpha \leq 1$.  
Then by \eqref{lambdaGEQ: e}, 
\begin{align*}
\qdeg \mu &\geq (p^e \tr{\lambda}{e} + 1) \cdot \qdeg f \\
& \geq \( p^e \(\lambda - \frac{1}{p^e}\) + 1 \) \cdot \qdeg f \\
& = p^e \cdot \lambda \cdot \qdeg f \\
&= p^e \cdot \sum \qdeg x_i. 
\end{align*}
We may now conclude that $\mu \in \bracket{\m}{e}$ Lemma \ref{boundNotInBracket: L}.

Now say that $\fpt{f} = \tr{\lambda}{\Ell}$ and $(p^{\Ell} - 1) \cdot \lambda \in \N$ for some $L\leq 1$.
By Lemma \ref{fptfg<fptf: L}, it suffices to show that $(f+g)^{p^L \cdot \fpt{f}} \in \bracket{\m}{\Ell}$. 
Indeed, $p^L \cdot \fpt{f} > p^L \tr{\fpt{f}}{L} = \new{L}$ (the equality by Proposition \ref{Truncation: P}), so that $f^{p^L \cdot \fpt{f}} \in \bracket{\m}{L}$;
thus, $(f+g)^{p^L \cdot \fpt{f}} \equiv f^{p^L \cdot \fpt{f}} \equiv 0 \bmod \bracket{\m}{L}$ by Lemma \ref{SubtleExponent: L}.
\end{proof}

We see that the hypotheses of Proposition \ref{higherdegreefpt: P} are often satisfied in Example \ref{ExtraMonomialCounterex: E} below.  We also see that the statement of the proposition is sharp in the sense that there exist $f$ and $g$ satisfying its hypotheses such that $\fpt{f} = \tr{\lambda}{\Ell}$ for some $L\geq 1$, $(p^{\Ell} - 1)  \cdot  \lambda \notin \N$, and $\fpt{f+g} > \fpt{f}$.

\begin{example}
\label{ExtraMonomialCounterex: E}

Let $f = x^{15}+x y^7 \in \LL[x,y]$, which is \qh with $\qdeg f = 15$ under the grading determined by $(\qdeg x, \qdeg y) = (1,2)$, and has an isolated singularity at the origin when $p \geq 11$.  It follows from Theorem \ref{Main2D: T}
that \[ \fpt{f} = \tr{ \frac{1+2}{15}}{L} = \tr{\frac{1}{5}}{L}, \] where $1 \leq L \leq \order{p}{5} \leq 4$, or $L = \infty$ (i.e., $\fpt{f} = \frac{1}{5}$).  Furthermore, as $f$ is a binomial, we can use the algorithm given in \cite{Binomials}, and recently implemented by Sara Malec, Karl Schwede, and the third author in an upcoming {\tt{Macaulay2}} package, to compute the exact value $\fpt{f}$, and hence, the exact value of $L$ for a fixed $p$.  We list some of these computations in Figure \ref{table}.

\begin{figure}[h!]
\begin{center}
\begin{minipage}[t]{5cm}
\begin{tabular}{| c | c | c | c |}  \hline
$p$ & $L$ & $(p^L-1) \cdot \frac{1}{5} \in \N?$ \\  \hline
11 & 1 & Yes \\
13 & 1 & No \\
$17$ & $1$ & No \\ 
$19$ & $2$ & Yes \\ 
23 & 4 & Yes \\
29 & $\infty$ & -- \\
31 & 1 & Yes \\
\hline
\end{tabular}
\end{minipage}
\hfill
\begin{minipage}[t]{5cm}
\begin{tabular}{| c | c | c | c |}  \hline
$p$ & $L$ & $(p^L-1) \cdot \frac{1}{5} \in \N?$ \\  \hline
$37$ & $4$ & Yes \\ 
$41$ & $1$ & Yes \\ 
$43$ & $\infty$ & -- \\
$47$ &  1 & No \\ 
53 & 4 & Yes \\
59 & 2 & Yes \\ 
61 & 1 & Yes \\
\hline
\end{tabular}
\end{minipage}
\hfill
\begin{minipage}[t]{5cm}
\begin{tabular}{| c | c | c | c |}  \hline
$p$ & $L$ & $(p^L-1) \cdot \frac{1}{5} \in \N?$ \\  \hline
67 & 1 & No \\
$71$ & $\infty$ & -- \\ 
73 & $3$ & No \\
79 & 2 & Yes \\
83 & 2 & No \\
97 & 1 & No \\
101 & 1 & Yes \\
 \hline
\end{tabular}
\end{minipage}
\end{center}
\label{table}
\caption{Some data on $F$-pure thresholds of $f=x^{15} + xy^7 \in \LL[x,y]$.}
\end{figure}

We see that the hypotheses of Proposition \ref{higherdegreefpt: P} are often satisfied in this example, and it follows that $\fpt{f} = \fpt{f+g}$ for every $g \in \Deg{R}{\geq 16}$ whenever either ``$\infty$'' appears in the second column or ``Yes'' appears in the third.  
When $p=17$, however, we have that 
$\fpt{f} = \tr{\frac{1}{5}}{1} = \frac{3}{17}$, and 
when $g \in \set{x^{14} y, x^{12}y^2, y^8,x^{13}y^2, x^{14}y^2  } \subseteq \Deg{R}{\geq 16}$, 
one may verify that $(f+g)^3 \notin {\m}^{[17]}$, so that 
 it follows from Lemma \ref{fptfg<fptf: L} that $\fpt{f+g} > \frac{3}{17}$.
For another example of this behavior, it can be computed that when $p = 47$ and $g \in \set{x^{12}y^2, x^{10}y^3, x^8y^4, x^4y^6,  x^9y^4, x^{10}y^4  }$, then $\fpt{f+g}>\fpt{f}$.
\end{example}

We now present the main result of this subsection. 
\begin{proposition}
\label{madicusc: P}  
Suppose that  $f \in \mathbb{F}_p[x_1, \cdots, x_n]$ is \qh under some $\NN$-grading such that $\sqrt{\jac{f}} =(x_1, \ldots, x_n)$ and ${\qdeg f}\geq {\sum \qdeg x_i}$. Then $\fpt{f+g}=\fpt{f}$ for each $g \in \Deg{R}{\geq n \qdeg f - \sum \qdeg x_i+1}$.
\end{proposition}

\begin{proof}
Let $\lambda = \frac{\sum \qdeg x_i}{\qdeg f}$.
If $\fpt{f} = \lambda$, then Proposition \ref{higherdegreefpt: P} implies that $\fpt{f}= \fpt{f+g}$. For the remainder of this proof, we will assume that $\fpt{f} \neq \lambda$.  By Proposition \ref{fptf<fptfg: P}, it suffices to show that $\fpt{f} \geq \fpt{f+g}$.  
Since $\fpt{f} = \tr{\lambda}{L} - \frac{E}{p^L}$ for some integers $E \geq 0$ and $L \geq 1$ by Theorem \ref{Main: T}, it is enough to show that $(f+g)^{p^L \fpt{f}} \in \bracket{\m}{\Ell}$ by Lemma \ref{fptfg<fptf: L}. 

To this end, note that 
\begin{equation}
 \label{crucialinequality: e}
 \fpt{f} = \tr{\lambda}{\Ell} - \frac{E}{p^L} \geq \lambda - \frac{1}{p^L} - \frac{E}{p^L} \geq \lambda - \frac{n-1}{p^L}, 
 \end{equation}
where the first inequality follows from Lemma \ref{BasicProperties: L}, and the second from our bounds on $E$.  
Suppose, by way of contradiction, that $(f+g)^{p^L \cdot \fpt{f}} \notin \bracket{\m}{\Ell}$.  As \[ (f+g)^{p^L \cdot \fpt{f}} = f^{p^L \cdot \fpt{f}} + \sum_{k=1}^{p^L \cdot \fpt{f}}  \binom{p^L \cdot \fpt{f}}{k}  f^{p^L \cdot \fpt{f} - k}  g^{k}, \] the inequality $p^L \cdot \fpt{f} > p^L \tr{\fpt{f}}{\Ell} = \new{\Ell}$ implies that $f^{p^L \cdot \fpt{f}} \in \bracket{\m}{\Ell}$, and so there must exist $1 \leq k \leq p^L \cdot \fpt{f}$ for which $f^{p^L \cdot\fpt{f}-k}  g^k \notin \bracket{\m}{\Ell}$. We will now show, as in the proof of Lemma \ref{SubtleExponent: L}, that this is impossible for degree reasons.  Indeed, for such a $k$, there exists a supporting monomial $\mu$ of $f^{p^L \fpt{f} - k} g^k$ not contained in $\bracket{\m}{\Ell}$, so that $\qdeg \mu \leq (p^L - 1) \cdot \sum \qdeg x_i$ by Lemma \ref{boundNotInBracket: L}. However, as $g \in \Deg{R}{\geq n \cdot \qdeg f - \sum \qdeg x_i + 1}$, 
\begin{equation} \label{deriv: e} \qdeg \mu  \geq \qdeg f \cdot (p^L \cdot \fpt{f} - k) + k \cdot \( n \cdot \qdeg f - \sum \qdeg x_i + 1 \). \end{equation}
The derivative with respect to $k$ of the right-hand side of \eqref{deriv: e} is $(n-1) \qdeg f - \sum \qdeg x_i + 1$, which is always nonnegative by our assumption that $\qdeg f \geq \sum \qdeg x_i$.  Thus, the right hand side of \eqref{deriv: e} is increasing with respect to $k$, and as $k \geq 1$, 
\begin{align*}
\qdeg \mu & \geq \qdeg f \cdot (p^L \cdot \fpt{f} - 1) + \(n \cdot \qdeg f - \sum \qdeg x_i + 1\) \\
	        & \geq \qdeg f \cdot (p^L\cdot \lambda - n) + \( n \cdot \qdeg f - \sum \qdeg x_i + 1 \) \\ 
                & = p^L \cdot \qdeg f \cdot \lambda - \sum \qdeg x_i + 1 = \( p^L - 1 \) \cdot \sum \qdeg x_i + 1,
\end{align*}
where the second inequality above is a consequence of \eqref{crucialinequality: e}.  Thus, we have arrived at a contradiction, and we conclude that $(f+g)^{p^L \cdot \fpt{f}} \in \bracket{\m}{\Ell}$, completing the proof.
\end{proof}

\bibliographystyle{alpha}
\bibliography{refs}

\vspace{.5cm}

{\tiny
\noindent \small \textsc{Department of Mathematics, University of Utah, Salt Lake City, UT 84112} 

\emph{Email address}:  \href{mailto:dhernan@math.utah.edu}{\tt dhernan@math.utah.edu} 

\vspace{.25cm}

\noindent \small \textsc{Department of Mathematics, University of Virginia, Charlottesville, VA 22904} 

\emph{Email address}:  \href{mailto:lcn8m@virginia.edu}{\tt lcn8m@virginia.edu} 

\vspace{.25cm}

\noindent \small \textsc{Department of Mathematics, University of Minnesota, Minneapolis, MN  55455} 

\emph{Email address}:  \href{mailto:ewitt@umn.edu}{\tt ewitt@umn.edu} 

\vspace{.25cm}

\noindent \small \textsc{Department of Mathematics, University of Nebraska, Lincoln, NE  68588} 

\emph{Email address}:  \href{mailto:wzhang15@unl.edu}{\tt wzhang15@unl.edu} 
}
\end{document}